\documentclass[12pt]{amsart}
\usepackage[utf8]{inputenc} 
\usepackage[active]{srcltx}
\usepackage{a4wide}
\usepackage{amsthm,amsfonts,amsmath,mathrsfs,amssymb,bm}
\usepackage{dsfont}
\usepackage{tikz}
\usepackage{mathtools}
\usepackage[T1]{fontenc}
\usepackage{enumerate}
\usepackage{comment}
\usepackage[left=4cm,top=4cm,right=4cm,bottom=4cm]{geometry}
\numberwithin{equation}{section}

\usepackage{pgf,tikz}
\usepackage{pgfplots}

\usepackage{graphicx}
\usepackage{amsmath}
\usepackage{amsthm}
\usepackage{amssymb}
\usepackage{amscd}
\usepackage{color}
\usepackage{listings}
\usepackage[colorlinks=true,urlcolor=blue,citecolor=blue,linkcolor=blue]{hyperref}
\usepackage{esint}
\usepackage{xcolor}
\usepackage{authblk}

\usepackage{fancyhdr}

\usepackage{etex}
\usepackage[all]{xy}
\usepackage{pgf,tikz}
\usetikzlibrary{arrows}
\usetikzlibrary{patterns}
\def\a{\alpha}       \def\b{\beta}        
            
\def\s{\sigma}              
                  \def\z{\zeta}

           \def\O{\Omega}

\newcommand{\R}{{\mathbb R}} 
\newcommand{\C}{{\mathbb C}} 
\newcommand{\N}{{\mathbb N}}

\newcommand{\T}{{\mathbb T}} 
\newcommand{\D}{{\mathbb D}}

\newcommand{\Ha}{{\mathbb H}}
\newcommand{\Fi}{\varphi}
\newcommand{\al}{\alpha}
\newcommand{\esp}{{\sigma_p(\Delta\mid_{H^p})}}
\newcommand{\dom}{\textnormal{dom}}
\renewcommand{\Re}{\mathrm{Re}}
\renewcommand{\Im}{\mathrm{Im}}

\newcommand{\norm}[1]{\left\lvert\left\lvert#1\right\rvert\right\rvert}

\DeclareMathOperator*{\PR}{Re}
\DeclareMathOperator*{\Impart}{Im}

\newcommand{\coinf}{\textnormal{co}_\infty}

\newtheorem{theorem}{Theorem}[section]
\newtheorem{lemma}[theorem]{Lemma}
\newtheorem{proposition}[theorem]{Proposition}
\newtheorem{corollary}[theorem]{Corollary}

\theoremstyle{definition}
\newtheorem{definition}[theorem]{Definition}
\newtheorem{example}[theorem]{Example}

\theoremstyle{remark}
\newtheorem{remark}[theorem]{Remark}

\numberwithin{equation}{section}

\date{\today}

\title{On frequencies of parabolic Koenigs domains}
\author[C. G\'omez-Cabello ]{Carlos G\'omez-Cabello}
\address{Carlos Gómez-Cabello 
\newline Department of Mathematics University of the Basque
	Country/Euskal Herriko Unibertsitatea, 48490
Leioa, Spain.
}
\email{carlos.gomezca@ehu.eus}

\author[F. J. González-Doña]{F. Javier González-Doña}
\address{F. Javier González-Doña 
\newline Universidad de Sevilla, Departamento de Matemática Aplicada II, Escuela Técnica Superior de Ingeniería (US), Camino de los Descubrimientos s/n, 41092, Sevilla.}
\email{fgonzalez13@us.es}

\begin{document}
\def\tempAuthorList{\small{CARLOS GÓMEZ-CABELLO AND F. JAVIER GONZÁLEZ-DOÑA}}

\thanks{First author is supported by the
Basque Government through grant IT1615-22. Second author is partially supported by Plan Nacional  I+D grant no. PID2022-137294NB-I00, Spain,
		the Spanish Ministry of Science and Innovation,  through the ``Ayuda extraordinaria a Centros de Excelencia Severo Ochoa'' (20205CEX001) and from `Ayudas para Jóvenes Doctores', Plan Propio de Investigación, Universidad Carlos III de Madrid, 2024/00724/001}
	
	%    General info
	
    \maketitle
\begin{center}
\tempAuthorList
\end{center}
\begin{abstract}
Let $(\varphi_t)_{t\geq 0}$ be a parabolic semigroup of analytic functions on $\D$, with Koenigs function $h$ and Koenigs domain $\Omega = h(\D)$. We study the point spectrum $\esp$ of $\Delta$, the infinitesimal generator  of the $C_0$-semigroup $(C_{\varphi_t})_{t\geq 0}$ of composition operators on $H^p$. This reduces to characterizing the frequencies of $\Omega$. That is, those $\lambda \in \C$ such that $e^{\lambda h} \in H^p$. We derive containment relations for $\esp$ and provide sufficient conditions for its complete characterization. Our approach relies heavily on the geometric properties of $\Omega$ and on careful estimates of the harmonic measure of some boundary subsets of $\Omega$.  We conclude with some consequences regarding the spectrum of the composition operators $(C_{\varphi_t})_{t\geq 0}$. These results extend a previous work of Betsakos on hyperbolic semigroups.
\end{abstract}

\tableofcontents

\section{Introduction}
Let $X$ be a Banach space of analytic functions in the unit disc $\D$. The composition operator $C_{\varphi}$ of analytic \emph{symbol} $\varphi:\D\to\D$ acting on $X$ is defined by $C_{\varphi}f=f\circ\varphi$, $f\in X$. These operators have historically attracted, and continue to attract, special attention in the field of Complex Analysis and Operator Theory. Among the many interesting questions arising in relation to these operators, the study of the spectrum and the point spectrum has played a prominent role. In this work, we will focus on the case where $X$ is any Hardy space $H^p$ of the unit disc, with $1\leq p<\infty.$ We say that a holomorphic function $f:\D\to\C$ belongs to the Hardy space $H^p$, $0<p<\infty$, if
\begin{equation*}
    \|f\|_{H^p}:=\sup_{0\leq r<1}\left( \frac{1}{2\pi}
     \int_0^{2\pi}|f(re^{i\theta})|^pd\theta
    \right)^{\frac1p}<\infty.
\end{equation*}
It is well known that for $1\leq p<\infty$, $\|\cdot\|_{H^p}$ defines a norm, whereas for $0<p<1$, it is a quasi-norm.

\medskip

The point spectrum of composition operators on $H^p$ spaces, which is characterized via  Schroeder's equation

\begin{equation}\label{Schroder equation}
    f\circ \Fi = \lambda f \qquad (f \in H^p\setminus\{0\}, \lambda \in \C),
\end{equation}
is completely understood when the symbol $\varphi$ is a linear fractional map (\cite{CS,Cowen, Deddens, Kamowitz, Nordgren}). In the general setting, despite the existence of some partial results (see the monograph \cite{CM}), we lack a complete characterization of the point spectrum.

\medskip

    One particularly interesting situation is that of composition operators on $H^p$  embedded into a $C_0$-semigroup $(C_{\Fi_t})_{t\geq 0}$. By the seminal work of Berkson and Porta \cite{BP}, these correspond exactly to the symbols $(\Fi_t)_{t\geq 0}$ that form a continuous semigroup of analytic functions on $\D$. 

    The point spectrum of these composition operators can be characterized in terms of the point spectrum of the infinitesimal generator $\Delta$ of the $C_0$-semigroup (see \cite{EN}). Indeed,
    \begin{equation}\label{spectral mapping theorem point spectrum}
		\sigma_p(C_{\Fi_t}\mid_{H^p})  = \{ e^{t\lambda}: \lambda\in \sigma_p(\Delta\mid_{H^p}) \} \qquad (t>0).
	\end{equation}
	Moreover, Siskakis \cite{Tesis_Siskakis} showed that $\lambda \in \sigma_p(\Delta\mid_{H^p})$ if and only if $e^{\lambda h} \in H^p$, where $h: \D\rightarrow \C$ is the Koenigs function of the semigroup $(\Fi_t)_{t\geq 0}$. Under these considerations, in 2017, Betsakos characterized \cite{Betsakos} the point spectrum of infinitesimal generators of hyperbolic composition semigroups on $H^p$ in terms of the geometry of the associated Koenigs domain $\Omega = h(\D)$. However, the parabolic case remained open.

    \medskip

    In this sense, our main goal is to tackle the problem of characterizing the point spectrum of parabolic composition semigroups on $H^p$. This, quoting Betsakos \cite{Betsakos}, ``is probably much more difficult to describe'' than in the hyperbolic case. As discussed above, such a description is equivalent to characterizing those $\lambda \in \C$ such that $e^{\lambda h} \in H^p$.

    \medskip

    It is important to point out that this problem is equivalent to studying which exponentials $e^{\lambda z}$ belong to the Hardy space $H^p(\Omega)$, which is the space of holomorphic functions $f:\Omega \rightarrow \C$ such that $|f|^p$ has a harmonic majorant in $\Omega$. Those $\lambda \in \C$ satisfying that $e^{\lambda z}\in H^p(\Omega)$ are called the \textit{frequencies} of $\Omega$. In a very recent paper \cite{BGY}, Bracci, Gallardo-Gutiérrez, and Yakubovich have characterized whether the linear manifold spanned by the $e^{\lambda z}\in H^\infty(\Omega)$ is
    weak-$*$ dense in $H^\infty(\Omega)$ in terms of geometric conditions on the Koenigs domain $\Omega$. In their work, the authors obtain sufficient conditions for the norm-density of such a manifold in $H^p$.

    \medskip

    The methods developed in this text are largely based on careful estimates of the harmonic measure of some boundary subsets of $\Omega$. In this sense, these techniques are related to both classical and recent results on diverse problems. For instance, the behaviour of the orbits of semigroups on $\D$ (see \cite{BeCD} or the most recent \cite{BCZ}), the characterization of the Hardy number of Koenigs domains \cite{CCKR}, as well as the computation of the essential spectral radius of composition operators on $H^2$ \cite{Poggi-Corradini}.

    \medskip

It is also worth mentioning the existence of another recent work \cite{kourouetal} on the topic. There, the authors study the problem in several Banach spaces of analytic functions in the unit disc, especially focusing on the Dirichlet space. The results they prove for Hardy spaces $H^p$ correspond to Theorem \ref{teorema contenciones general} in this manuscript. We focus our attention on $H^p$ spaces and we are concerned with describing $\esp$ in terms of the geometric features of the domain $\Omega$.

    \medskip

    The manuscript is organized as follows:
    
    \medskip

In Section \ref{seccion toolbox}, we develop some technical tools that will play a central role in the remainder of the text. First, we solve the problem for the specific situation of $\O$ being an angular sector, and present some general results regarding $\esp$. Later, we provide sufficient geometric conditions on $\Omega$ to assure the membership of $e^{\lambda h}$ to the Smirnov class $\mathcal{N}^+$, which will allow us to express its norm in terms of harmonic measures (Corollary \ref{corolario formula norma}). We finish by providing explicit computations of the harmonic measure of some particular boundary subsets of angular sectors and vertical strips, which will be needed in the forthcoming sections. 

    \medskip

  Section \ref{seccion general} is devoted to general parabolic Koenigs domains $\Omega$. Under the assumption of $\Omega$ being contained in an angular sector of opening angle less than or equal to $\pi$, we introduce the concepts of maximal angles. These notions will allow us to state some containment relations for $\esp$ (Theorem \ref{teorema contenciones general}). In order to study the membership to $\esp$ of those $\lambda\in \C$ not covered by this result, we study the boundary of $\Omega$ through the defining functions of the starlike at infinity hull of its rotations (see Definition \ref{definicion hull} and the discussion afterwards). Making use of these new concepts, together with estimates involving  harmonic measure, we derive sufficient conditions for $\lambda \in \esp$ (Theorem \ref{teorema no convexo por arriba} and Theorem \ref{teorema no convexo por arriba no contiene sector}). These results allow us to completely characterize $\esp$ in various situations (Corollary \ref{corolario no convexo contiene sector}, Corollary \ref{corolario no convexo no contiene sector}). Moreover, by adding a convexity assumption on $\Omega$, we obtain necessary conditions that allow us to characterize $\esp$ for all situations, up to two points (Subsection \ref{subseccion convexo}).

\medskip

Finally, Section \ref{seccion espectro} focuses on studying the spectrum of the associated composition operators $(C_{\Fi_t})_{t\geq 0}$ on $H^p$, $1\leq p<\infty$. Under some geometric assumptions on $\Omega$, we are able to completely characterize $\sigma(C_{\Fi_t}\mid_{H^p})$ for all $t>0$ (Theorem \ref{teorema espectro}). In particular, this result may be applied to a subclass of composition operators with symbols of positive hyperbolic step (Corollary \ref{corolario paso hiperbolico}).

\medskip

We conclude this introductory section by recalling some notions on semigroups of analytic functions on $\D$, hyperbolic distance, and harmonic measure that will be needed through the exposition.
\medskip

	\subsection{Preliminaries}
	
	As pointed out previously, $(C_{\Fi_t})_{t\geq 0}$ is a $C_0$-semigroup acting on $H^p$, $1\leq p < \infty$ if and only if the family of symbols $(\Fi_t)_{t\geq 0}$ is a continuous semigroup of analytic functions on $\D$. For this to occur, the following three conditions must hold:
	\begin{enumerate}
		\item [(i)] $\Fi_0(z) = z$ for every $z \in \D.$
		\item [(ii)] $\Fi_{t+s} = \Fi_t \circ \Fi_s$ for every $t,s \geq 0.$
		\item [(iii)] If $t\rightarrow s,$ then $\Fi_t \rightarrow \Fi_s$ uniformly on compact subsets of $\D$ for $t,s\geq 0$.
	\end{enumerate}

	\medskip
	
	Let $(\Fi_t)_{t\geq 0}$ be a continuous semigroup which does not contain automorphisms of $\D$ with a fixed point in $\D$. In such a case, there exists a unique point $\tau\in \overline \D$ such that $(\Fi_t)_{t\geq 0}$ converges uniformly on compacta, as $t$ goes to $+\infty$, to the constant map $z\mapsto \tau$. The point $\tau$ is called the {\sl Denjoy-Wolff point} of $(\Fi_t)_{t\geq 0}$. This important feature of continuous semigroups on $\D$ allows us to classify them depending on the location of the Denjoy-Wolff point. If $\tau\in\D$, the semigroup is said to be \emph{elliptic}. If $\tau\in\T:=\partial\D$, we say that the semigroup is \emph{non-elliptic}. It is known that the functions $\Fi_t'$ have non-tangential limits $e^{t\al}$ at the Denjoy-Wolff point $\tau,$ with $\al \leq 0$. This allows a further classification of non-elliptic semigroups. Indeed, $(\Fi_t)_{t\geq 0}$ is said to be \textit{hyperbolic} if $\al < 0$ and \textit{parabolic} if $\al = 0.$
	
	\medskip
	
	We will focus our attention on non-elliptic semigroups. If $(\Fi_t)_{t\geq 0}$ is non-elliptic, there exists a univalent function $h:\D\rightarrow \C$ satisfying
	\begin{equation}\label{koenigs function}
		\Fi_t(z) = h^{-1}(h(z)+t) \qquad (z\in \D, \ t\geq 0).
	\end{equation}
The function $h$ is unique, up to an additive constant. We refer to the function $h$ as the \textit{Koenigs function} of the semigroup, see  \cite[Theorem 9.3.5]{BCD}. Its image, $\Omega := h(\D)$, is the \textit{Koenigs domain} of the semigroup. This domain is \textit{starlike at infinity}, meaning that, for every $z\in \Omega$ and every $t \geq 0$, $z+t \in \Omega.$ Throughout the paper, $\Omega$ will always stand for the Koenigs domain  of the semigroup $(\Fi_t)_{t\geq 0}$.  
	\medskip
	
We will use a recent fruitful description of the Koenigs domain. Indeed, every Koenigs domain $\Omega$ can be described as \label{pagina defining functions}
	$$\{x+iy \in \C: y \in I, \ x>\psi_\Omega(y)\},$$ where $\psi_\Omega :I \rightarrow [-\infty,+\infty)$ is an upper semicontinuous function and $I\subset (-\infty,\infty)$ is an open connected interval. This function is known as the \textit{defining function} of $\Omega$ (see \cite[Definition 2.3]{BGY}).
	
	\medskip
	
	The geometry of the Koenigs domain encodes valuable information about the semigroup. In this respect, a semigroup $(\Fi_t)_{t\geq 0}$ is parabolic if and only if its Koenigs domain is not contained in a horizontal strip \cite{CD}. Thus, $\Omega$ is parabolic if and only if $\dom(\psi_\Omega)$ is not bounded. We will call \textit{parabolic Koenigs domain} a Koenigs domain not contained in a horizontal strip.
	
	\medskip

    It is standard in the literature to distinguish two classes of parabolic semigroups.  A parabolic semigroup $(\Fi_t)_{t\geq 0}$ is said to be of \textit{positive hyperbolic step} if for some $z \in \D$ the quantity $$\lim\limits_{t\rightarrow +\infty} k_\D(\Fi_t(z),\Fi_{t+1}(z))$$ is positive, where $k_\D$ denotes the hyperbolic distance in $\D$. If the previous situation does not hold, $(\Fi_t)_{t\geq 0}$ is of \textit{zero hyperbolic step}. This property can also be characterized through the geometry of $\Omega$. In fact, $(\Fi_t)_{t\geq 0}$ is of positive hyperbolic step if and only if $\Omega$ is contained in a horizontal half-plane \cite[Theorem 9.3.5 and Proposition 9.3.10]{BCD}.

There exists another object closely linked to the continuous semigroup: the \emph{infinitesimal generator}.  A remarkable result of Berkson and Porta \cite{BP} asserts that each continuous semigroup of holomorphic self-maps of $\D$ is locally uniformly differentiable with respect to the parameter $t\geq 0.$  Not only this, but the function $t\mapsto\Fi _{t}(z)$, $z\in\D$, is the solution of the Cauchy problem:
\begin{equation}\label{cauchy}
\frac{\partial \Fi_t(z)}{\partial t} = G(\Fi_t(z))\quad
\mbox{and} \quad \Fi _{0}(z)= z\in \D.
\end{equation}
The function $G$ is called the {\sl infinitesimal generator} of
the semigroup $( \Fi _{t})_{t\geq0}$. There exists an intimate connection between the infinitesimal generators of the semigroup of symbols $( \Fi _{t})_{t\geq0}$ and the semigroup of operators. In fact, if $(C_{\Fi_t})_{t\geq 0}$ is the associated $C_0-$semigroup, then
	\begin{equation}\label{expresion generador infinitesimal}
		\Delta f  = G\cdot f' \qquad (f\in H^p)
	\end{equation}	is its 	infinitesimal generator, which, we recall, is defined as
    $$
    \Delta f = \lim\limits_{t\to 0^+} \frac{C_{\Fi_t}f-f}{t} \qquad (f \in \dom(\Delta)).
    $$
    It turns out that $\Delta$ is a closed densely defined linear operator acting on $H^p$.

\medskip 

In what follows, unless stated otherwise, given $(\Fi_t)_{t\geq 0}$ a parabolic semigroup on $\D$, we will denote by $h$  its associated Koenigs function and by $\Omega$ its associated Koenigs domain. We will assume without loss of generality that $0\in \Omega$. Likewise, we will denote by $\Delta$ the infinitesimal generator of the $C_0$-semigroup $(C_{\Fi_t})_{t\geq 0}$ on $H^p$, $1\leq p< \infty$.

\medskip

Another useful tool will be the hyperbolic distance. Given a hyperbolic planar domain $\Omega$, we denote by $k_{\Omega}(z,w)$, $z,w\in\Omega$, the \textit{hyperbolic distance} in $\Omega$ between $z,w\in\Omega$. A simply connected domain is an example of a hyperbolic planar domain. Since we will be dealing with this particular class of domains, the conformal invariance of the hyperbolic distance allows us to compute the hyperbolic distance in a wide range of simply connected domains, such as the right half-plane $\C_+$, where
	\begin{equation*}
		k_{\C_+}(z,w)=\frac12\log
		\left(
		\frac{1+|T_w(z)|}{1-|T_w(z)|}
		\right),\quad z,w\in\C_+,
	\end{equation*}
	with $T_w(z)=(z-w)/(z+\overline{w})$. Another property which will be used in the course of the proofs is the continuity of the hyperbolic distance (\cite[Proposition 1.3.14]{BCD}).

    \medskip

	Harmonic measure will play a crucial role in the proof of the main results of this work. We now recall its definition, and some of its basic properties. For a more detailed exposition on the topic, the reader is referred to \cite{Ran}, \cite{GaMA}, or \cite{BCD} for simply connected domains.
	
	\medskip
    
	Let $\Omega$ be a domain in the complex plane whose boundary has positive logarithmic capacity, i.e., it is non-polar. Consider $E$ a Borel subset of $\partial_\infty \Omega$. The harmonic measure of $E$ relative to $\Omega$ is the generalized Perron-Wiener solution $u$ of the Dirichlet problem for the Laplacian in $\Omega$, where the boundary values are given by the characteristic function of the set $E$. We will use the standard notation
	\begin{equation*}\label{hm1}
		u(z)=\omega(z,E,\Omega),\quad z\in \Omega.
	\end{equation*}
	In our setting, we will deal only with simply connected domains. The complement of every simply connected domain $\Omega$ contains a continuum, which is a non-polar set \cite[Corollary 3.8.5]{Ran}. Then, the harmonic measure is well defined for these domains.
	
	\medskip
    
	A  nice property of harmonic measure is its conformal invariance (see \cite[Theorem 4.3.8]{Ran}). Another useful tool is Harnack's inequality. Indeed,	it allows one to relate the value of the harmonic function $\omega(\cdot,E,\Omega)$ at two different points $z,w\in\Omega$ (\cite[Theorem 1.3.1]{Ran}). In the setting of simply connected domains, we can explicitly compute the constant appearing in Harnack's inequality, which takes the following form (see \cite[Section 1.3]{Ran} and \cite[Corollary 4.3.5]{Ran}):
	\begin{theorem}\label{teorema distancia hiperbolica}
		Let $\Omega\subsetneq\C$ be a simply connected domain and let $E$ be a Borel subset of $\partial_{\infty}\Omega$. Then, for every $z_1,z_2\in\Omega$
		\[
		\omega(z_1,E,\Omega)\leq e^{2k_{\Omega}(z_1,z_2)}\omega(z_2,E,\Omega).
		\]
	\end{theorem}

		We will use the notation $f(x)\lesssim g(x)$ if there is some constant $C>0$ such that $|f(x)|\leq C|g(x)|$ for all $x$. If we have simultaneously $f(x)\lesssim g(x)$ and $g(x)\lesssim f(x)$, we write $f\approx g$. The use of this notation will imply that the underlying constants are absolute, which means that they do not depend on any other magnitude. 

    \     
    Also, given $\theta\in\R$, we define $\C_{\theta}:=\{z\in\C: \ \Re(z)>\theta\}$ and $\C_+=\C_0$.

\section{Toolbox: angular sectors, Smirnov class, and harmonic measure }\label{seccion toolbox}
 For the sake of illustrating some of the main ideas that will appear in Section \ref{seccion general}, we begin by characterizing $\esp$ for parabolic semigroups whose Koenigs domain $\Omega$ is an angular sector.

\subsection{The case of angular sector domains}	We begin by defining angular sectors. Given $z_0 \in \C$ and $\a,\b \in [0,\pi]$ with $\a+\b >0$, we define the angular sector centred on $z_0$, with lower angle $\al$ and upper angle $\beta$ as
$$S_{z_0}(\al,\beta) = z_0+\{re^{i\theta}: r>0, \ -\al < \theta < \beta \}. $$
To simplify the notation, we will write $S(\al,\beta):=S_0(\al,\beta)$, and $S_{z_0}(\al):=S_{z_0}(\al,\al)$.
	
	\medskip
	
	Recall that, for all $z_0$ and $\a,\b \in [0,\pi]$ with $\a+\b >0$, $S_{z_0}(\al,\beta)$ is a Koenigs domain, with associated Koenigs function  
	$$ h(z) = z_0 + e^{i\frac{\beta-\al}{2}}\left( \frac{1+z}{1-z}\right)^{\frac{\al+\beta}{\pi}} \qquad (z\in\D).$$
	
	\medskip
	
	The main result reads as follows:
	
	\begin{theorem}\label{espectro sector angular}
		Let $(\Fi_t)_{t\geq0}$ be a parabolic semigroup on $\D$. Assume that $\Omega = S_{z_0}(\al,\beta)$ for some $z_0\in \C$, $\al,\beta\in[0,\pi]$ and $\al+\beta > 0$. 
		\begin{enumerate}
			\item [(i)] If $\al+\beta \leq \pi$, then
			$$
			\sigma_p(\Delta\mid_{H^p}) = \left\{re^{i\theta}: r\geq 0, \ \frac{\pi}{2}+\al \leq \theta \leq \frac{3\pi}{2}-\beta \right\}.
			$$
			\item [(ii)] If $\al+\beta > \pi$, then
			$$\sigma_p(\Delta\mid_{H^p})=\{0\}.$$
		\end{enumerate}
	\end{theorem}
	\begin{remark}
		Notice that, if $\al+\beta = \pi$, then $S_{z_0}(\al,\beta)$ is a half-plane, which implies that $(\Fi_t)_{t\geq 0}$ is a semigroup of linear fractional maps. In fact, it is a semigroup of invertible elements (i.e., a group) if and only if $\al=\pi$ and $\beta=0$ or $\al=0$ and $\beta = \pi$ (see, for instance, \cite{Siskakis}).
	\end{remark}
	To prove Theorem \ref{espectro sector angular}, we will need the following results:
	
	\begin{proposition}\label{contencion espectro puntual}\cite[Section 6, Proposition 1]{Betsakos}
		Let $(\Fi_t)_{t\geq 0}$ and $(\tilde{\Fi}_t)_{t\geq 0}$  be two parabolic semigroups, with Koenigs domains $\Omega$ and $\Tilde{\Omega}$ and infinitesimal generators of the associated $C_0$-semigroups $\Delta$ and $\Tilde{\Delta}$, respectively. If $\tilde{\Omega}\subseteq \Omega$, then $\sigma_p(\Delta)\subseteq \sigma_p(\tilde{\Delta}).$
	\end{proposition}
	\begin{lemma}\label{lema giro}
		Let $(\Fi_t)_{t\geq 0}$ be a parabolic semigroup. Assume that there exists $e^{i\theta}\in \T$ such that $e^{i\theta}\Omega$ is also a Koenigs domain, and denote by $\Tilde{\Delta}$ the infinitesimal generator of the $C_0-$semigroup associated to $e^{i\theta}\Omega$. 
		Then, $$\sigma_p(\tilde{\Delta}) = e^{-i\theta}\sigma_p(\Delta).$$
	\end{lemma}
	\begin{proof}
		Recall that $\lambda \in \sigma_p(\tilde{\Delta})$ if and only if $e^{\lambda e^{i\theta}h} \in H^p$, so $\lambda \in \sigma_p(\tilde{\Delta})$ if and only if $e^{i\theta}\lambda \in \sigma_p(\Delta),$ which yields the result.
	\end{proof}
	
	Before proceeding with the proof of Theorem \ref{espectro sector angular}, note that, by \eqref{expresion generador infinitesimal}, $0$ is always an eigenvalue for $\Delta$. 
	
	\medskip
	
	\noindent \textit{Proof of Theorem \ref{espectro sector angular}.} First, we prove $(i)$. Assume that $\al+\beta \leq \pi$. We can assume, without loss of generality, that $z_0=0$. Moreover, by Lemma \ref{lema giro}, we can assume that $\al=\beta.$ At this point, the result follows as a byproduct of \cite[Corollary 6.2]{Cowen} and \eqref{spectral mapping theorem point spectrum}.
	
	To prove $(ii)$, assume that $\al+\beta >\pi$. Clearly, in such a case, $\Omega$ contains two different rotations of $\C_+$. Part $(i)$ and Proposition \ref{contencion espectro puntual} yield the result. \hfill $\Box$
	
	\medskip
	As a byproduct of Theorem \ref{espectro sector angular}  and Proposition \ref{contencion espectro puntual}, we deduce the following result:
	
	\begin{corollary}\label{corolario contenciones sector angular}
		Let $(\Fi_t)_{t>0}$ be a parabolic semigroup on $\D$. \begin{enumerate}
		 \item [(i)] Assume that there exist $z_0,z_1 \in \C$ and $\al,\beta\in[0,\pi],$ $0<\al+\beta\leq \pi$ such that
			$$S_{z_0}(\al,\beta) \subseteq \Omega \subseteq S_{z_1}(\al,\beta).$$ Then,
			$$ \sigma_p(\Delta\mid_{H^p}) = \left\{re^{i\theta}: r\geq 0, \ \frac{\pi}{2}+\al \leq \theta \leq \frac{3\pi}{2}-\beta \right\}. $$
			\item [(ii)] Assume that there exists $z_0 \in \C$, $\al,\beta \in  [0,\pi]$, $\al+\beta >\pi$ such that $ S_{z_0}(\al,\beta) \subseteq \Omega.$ Then, 			$$\sigma_p(\Delta\mid_{H^p}) = \{0\}. $$
		\end{enumerate}
	\end{corollary}
	Note that the point spectra described in Corollary \ref{corolario contenciones sector angular} consist exactly of those $\lambda \in \C$ such that $\lambda \Omega$ is contained in some left half-plane. As we will see, this property does not hold for every Koenigs domain $\Omega$, but at least, one of the two containments hold:
	
	\begin{proposition}\label{proposicion contenido semiplano}
		Let $(\Fi_t)_{t>0}$ be a parabolic semigroup on $\D$. Let $\lambda \in \C$ be such that $\lambda \Omega$ is contained in a left half-plane. Then, $e^{\lambda h} \in H^\infty$ and $\lambda \in \sigma_p(\Delta\mid_{H^p}).$ 
        \end{proposition}
	\begin{proof}
		Let $h:\D\rightarrow \Omega$ be the associated Koenigs map. If $\lambda\Omega$ is contained in a left half-plane, then there exists $\eta\in \R$ such that $\Re(\lambda h(z))< \eta$ for all $z\in\D$. Then,
		$$|e^{\lambda h(z)}| = e^{\Re(\lambda h(z))} \leq e^\eta$$ for all $z\in \D$, and $e^{\lambda h} \in H^\infty\subset H^p$. 
	\end{proof}
    The next result improves \cite[Lemma 4.3]{GD}:
    \begin{proposition}\label{proposicion convexidad}
        Let $(\Fi_t)_{t>0}$ be a parabolic semigroup on $\D$. Then, $\esp$ is convex.
    \end{proposition}
\begin{proof}
    Let $\lambda_1,\lambda_2\in\esp$. Take $\alpha\in[0,1]$ and consider the convex combination $\gamma_{\alpha}=\alpha\lambda_1+(1-\alpha)\lambda_2$. We show that $e^{\gamma_{\alpha}h}$ belongs to $H^p$, for every $\alpha\in[0,1]$. By Hölder's inequality with conjugated exponents $1/\alpha$ and $1/(1-\alpha)$, we have that, for $r\in[0,1)$,
    \begin{align*}
        \left(\int_0^{2\pi}|e^{(\alpha\lambda_1+(1-\alpha)\lambda_2)h(re^{i\theta})}|^pd\theta
        \right)^{1/p}
        &=
        \left(\int_0^{2\pi}
        e^{\alpha p\Re(\lambda_1h(re^{i\theta}))}
        e^{(1-\alpha) p\Re(\lambda_2h(re^{i\theta}))}
        d\theta
        \right)^{1/p}\\
        &
        \leq
        \left(
        \int_0^{2\pi}
        e^{p\Re(\lambda_1h(re^{i\theta}))}d\theta
        \right)^{\alpha/p}
        \left(
        \int_0^{2\pi}
        e^{p\Re(\lambda_2h(re^{i\theta}))}
        d\theta
        \right)^{(1-\alpha)/p}.
    \end{align*}
Taking the supremum on $r\in[0,1)$, the desired conclusion follows from both the definition of the $H^p$-norm and the assumption on $\lambda_1$ and $\lambda_2$.
\end{proof}

\subsection{Membership of $e^{\lambda h}$ to the Smirnov class $\mathcal{N}^+$}
We denote  $\log^+(x) = \max\{0,\log(x)\}$ for every $x>0$. A holomorphic function $f: \D\rightarrow \C$ belongs to the \textit{Nevanlinna class} $\mathcal{N}$ if
	$$ \sup\limits_{0<r<1} \frac{1}{2\pi} \int_0^{2\pi} \log^+|f(re^{i\theta})|d\theta < \infty. $$ It is well known that $H^p \subset \mathcal{N}$ for every $0<p<\infty$. Moreover, for every function $f \in \mathcal{N}$, its nontangential limit $f(e^{i\theta})$ exists almost everywhere \cite[Chapter 2]{Duren}.
	
	\medskip
	
	Let $f\in \mathcal{N}.$ We say that $f$ belongs to the \textit{Smirnov class} $\mathcal{N}^+$ if
	\begin{equation}\label{Smirnov}
	\lim\limits_{r\rightarrow 1^-} \frac{1}{2\pi} \int_0^{2\pi} \log^+|f(re^{i\theta})|d\theta = \frac{1}{2\pi} \int_0^{2\pi} \log^+|f(e^{i\theta})|d\theta.
	\end{equation}
	
	For functions in the Smirnov class, we will be able to study their membership to $H^p$ in terms of harmonic measure. Indeed, a function $f\in \mathcal{N^+}$ belongs to $H^p$ $(1\leq p < \infty)$ if and only if its nontangential limit belongs to $L^p(\T)$ \cite[Theorem 2.11]{Duren}. Thus, following \cite[p. 2256]{Betsakos} one has
	\begin{equation}\label{primera expresion medida armonica}
		\begin{split}
			\norm{f}^p_{H^p} & = \norm{f}^p_{L^p(\T)} = \frac{1}{2\pi}\int_0^{2\pi}|f(e^{i\theta})|^pd\theta
			\\& = \frac{p}{2\pi} \int_0^\infty r^{p-1}m(\{ \xi \in \T: |f(\xi)| > r\})dr
			\\ & = \frac{p}{2\pi} \int_0^\infty r^{p-1}\omega(0, \{ \xi \in \T: |f(\xi)| > r\}, \D) dr.
		\end{split}
	\end{equation}
	
	In order to apply \eqref{primera expresion medida armonica} to the functions $e^{\lambda h}$, we must provide sufficient conditions for them to belong to $\mathcal{N}^+:$
	
	\begin{proposition}\label{proposicion Smirnov}
		Let $h: \D\rightarrow \C$ be a univalent function. If $h \in H^1$, then $e^{\lambda h} \in \mathcal{N}^+$ for every $\lambda \in \C$. In particular, if $h(\D)$ is contained in $S_{z_0}(\a,\b)$ for some $z_0\in \C$, $\al,\beta\in [0,\pi],$ $0<\al+\beta<\pi$, then  $h\in H^1$ and $e^{\lambda h} \in \mathcal{N}^+$ for every $\lambda \in \C$.
	\end{proposition}
    
    \begin{proof}
    We begin by proving $e^{\lambda h} \in \mathcal{N}^+$ for every $\lambda \in \C$.
    For this purpose, let us see first that $e^{\lambda h}\in\mathcal{N}$. We have, for all $0<r<1,$
		\begin{align}\label{estimacionTCD}
		    \int_{0}^{2\pi} \log^+|e^{\lambda h(re^{i\theta})}| d\theta \leq \int_0^{2\pi} |\Re(\lambda h(re^{i\theta}))|d\theta &\leq |\lambda| \int_0^{2\pi} | h(re^{i\theta})| d\theta\notag \\&\leq 2\pi |\lambda|\norm{h}_{H^1} < \infty,
		\end{align} 
        giving the membership of $e^{\lambda h}$ to $\mathcal{N}$. Now, since $h\in H^1$, we have that the radial limits of $\log^+|e^{\lambda h(re^{i\theta})}|$ exist for a.e. $e^{i\theta}$. This, together with estimate \eqref{estimacionTCD} gives, by the Dominated Convergence Theorem, condition \eqref{Smirnov}. This proves the first part of the result. For the second one, assume that $h(\D)$ is contained in $S_{z_0}(\a,\b)$. We may assume, without loss of generality, that $z_0 = 0$. Consider $$\Tilde{h}(z) = e^{i\frac{\beta-\al}{2}}\left( \frac{1+z}{1-z}\right)^{\frac{\a+\b}{\pi}} \qquad (z\in \D).$$ It is straightforward to check that $\Tilde{h}$ is a univalent function mapping $\D$ onto $S(\a,\b).$ It is well known that the map $$z\in \D \mapsto \left( \frac{1+z}{1-z}\right)^\lambda$$ belongs to $H^1$ if and only if $|\Re(\lambda)|<1$, so $\Tilde{h}\in H^1$. Finally, let $\psi = \Tilde{h}^{-1}\circ h : \D\rightarrow \D$. By the Littlewood Subordination Principle, we have that $h\in H^1$, concluding the proof.
    \end{proof}

	\begin{corollary}\label{corolario formula norma}
		Let $(\Fi_t)_{t>0}$ be a parabolic semigroup on $\D$. Assume that $\Omega$ is contained in $S_{z_0}(\al,\beta)$ for some $z_0\in \C$, $\al,\beta\in[0,\pi],$ $0<\al+\beta<\pi$, and let $\lambda =|\lambda|e^{i\sigma}\in \C.$ Then
		\begin{equation}\label{expresion norma medida armonica}
			\norm{e^{\lambda h}}^p_{H^p} = \frac{p|\lambda|}{2\pi}\int_{-\infty}^\infty e^{p|\lambda|t} \omega(0,\{ \xi \in \partial\Omega_\sigma: \Re(\xi) >t\}, \Omega_\sigma)dt,
		\end{equation}  where $\Omega_\sigma = e^{i\sigma}\Omega.$
	\end{corollary}
	\begin{proof}
		First, recall that $e^{\lambda h}\in \mathcal{N}^+$ by Proposition \ref{proposicion Smirnov}, so by \eqref{primera expresion medida armonica}
		\begin{equation*}
			\begin{split}
				\norm{e^{\lambda h}}^p_{H^p} & = \frac{p}{2\pi} \int_0^\infty r^{p-1}\omega(0, \{ \xi \in \partial\D: |e^{\lambda h(\xi)}| > r\}, \D) dr
				\\ & = \frac{p}{2\pi} \int_0^\infty r^{p-1}\omega(0, \{ \xi \in \partial\D: \Re(\lambda h(\xi)) > \log r\}, \D) dr
				\\ & =  \frac{p}{2\pi} \int_0^\infty r^{p-1}\omega(0, \{ \xi \in \partial\Omega_\sigma: \Re(\xi) > \frac{\log r}{|\lambda|}\}, \Omega_\sigma) dr
				\\ & = \frac{p|\lambda|}{2\pi}\int_{-\infty}^\infty e^{p|\lambda|t} \omega(0,\{ \xi \in \partial\Omega_\sigma: \Re(\xi) >t\}, \Omega_\sigma)dt,
			\end{split}
		\end{equation*}
		where we have applied the conformal invariance of the harmonic measure  with the univalent map $e^{i\sigma} h : \D \rightarrow \Omega_\sigma$ in the third equality and the change of variable $\frac{\log r}{|\lambda|}=t$ in the fourth one. This finishes the proof.
	\end{proof}
    
\subsection{Computations of harmonic measure}
Let us introduce some notation. Given $z_0 \in \C$, $r_0>0$, and $\theta \in [-\pi,\pi]$, we define the half-lines 
	\begin{equation}\label{notacionlados}
	    \ell_{z_0,r_0}(\theta):=z_0 +\{ re^{i\theta}: r\geq r_0\}.
	\end{equation}
	If $r_0=0$, we simply write $\ell_{z_0}(\theta)$ and $\ell(\theta):=\ell_0(\theta)$.

    \ 
    We now prove some technical lemmas involving harmonic measure, which will be needed later in the discussion.
    \begin{proposition}\label{lematecnico}
		Let $\alpha\in[0,\pi/2)$ and $\beta\in[\pi/2,\pi)$ such that $\alpha+\beta\leq\pi$. Consider the sets $\widetilde{S}(\alpha,\beta)=\{z\in\C: \ \alpha<\emph{arg}(z)<\beta\}$ and
		$$
		E_t:=\{w\in\partial \widetilde{S}(\alpha,\beta): \ \emph{Re}(w)>t\},\quad t>1.
		$$
		Define $z_t:=(\partial \widetilde{S}(\alpha,\beta))\cap\{\emph{Re}(z)=t\}$ and $\varphi:=\beta-\al$. Then, if $z=|z|e^{i\theta_0}\in \widetilde{S}(\alpha,\beta)$, 
		\begin{equation}\label{pgeneral}
			\omega(z,E_t, \widetilde{S}(\alpha,\beta))=\frac{1}{2\pi}\arctan\left(\frac{2T}{T^2-1}\right),\quad \text{with} \quad 
			T=\frac{|z_t|^{\pi/\varphi}-|z|^{\pi/\varphi}\cos(\kappa)}{|z|^{\pi/\varphi}\sin(\kappa)},
		\end{equation}
		where $\kappa=(\theta_0-\alpha)\frac\pi\varphi$, for all $T>1$. In particular, if $\theta=(\alpha+\beta)/2$ and $z= re^{i\theta}$, $r>0$,
		\begin{equation}\label{pbisec}
			\omega(z,E_t,\widetilde{S}(\alpha,\beta))        
			=
			\frac{1}{2\pi}\arctan\left(
			\frac{2(|z_t|r)^{\pi/\varphi}}{|z_t|^{2\pi/\varphi}-r^{2\pi/\varphi}}
			\right)
		\end{equation} 
	for all $r>0$ such that $r<|z_t|$.
	\end{proposition}
	\begin{proof}
		Consider the Riemann map $F:\widetilde{S}(\alpha,\beta)\to \Ha $ given by $F(z)=(e^{-i\alpha}z)^{\pi/\varphi}$. By definition, $z_t=t+i\psi(t)$, for some $\psi(t)>0$. Then, we can rewrite the set $E_t$ as    
		\[
		E_t=\left\{z\in \partial \widetilde{S}(\alpha,\beta): \ |z|>\frac{t}{\cos(\alpha)}, \ \Re(z)>0\right\}\cdot
		\]
		Therefore,
		\[
		F_t:=F(E_t)=\left\{x\in\R: x>\left(\frac{t}{\cos(\alpha)}\right)^{\pi/\varphi}\right\}\cdot
		\]
		Now, since $\alpha=\arctan(\psi(t)/t)$, we have that $$\cos(\alpha)=\cos(\arctan(\psi(t)/t))=\frac{1}{\sqrt{\frac{\psi(t)^2}{t^2}+1}}=\frac{t}{|z_t|}\cdot$$
		Hence,
		$
		F_t=\left\{x\in\R: x>|z_t|^{\pi/\varphi}\right\}
		$. On the other hand, we have for $z=|z|e^{i\theta_0}\in \widetilde{S}(\alpha,\beta)$, letting $\kappa=(\theta_0-\alpha)\frac\pi\varphi$, 
		$$F(z)=e^{i(\theta_0-\alpha)\pi/\varphi}|z|^{\pi/\varphi}=|z|^{\pi/\varphi}(\cos(\kappa)+i\sin(\kappa)).$$ 
		Then, if
		\[
		G_t:=\left\{ 
		x\in\R:x>T
		\right\},\quad\text{with } T=\frac{|z_t|^{\pi/\varphi}-|z|^{\pi/\varphi}\cos(\kappa)}{|z|^{\pi/\varphi}\sin(\kappa)},
		\]
		by the conformal invariance of harmonic measure, we have by \cite[Example 7.2.5]{BCD}
		\begin{align*}%\label{calculofinal}
			\omega(z,E_t,\widetilde{S}(\alpha,\beta))
			=\omega(F(z),F_t,\Ha)=
			\omega(i,G_t,\Ha )
			=
			\frac{1}{2\pi}\arctan\left(\frac{2T}{T^2-1}
			\right),
		\end{align*}
		and \eqref{pgeneral} follows. 
		
		\medskip
        
		Take now $z=|z|e^{i\theta_0}\in\ell(\theta)$. Then, $\kappa=\pi/2$, so $F(z)=|z|^{\pi/\varphi}e^{i\kappa}
		=i|z|^{\pi/\varphi}$. Observe also that, $T=(|z_t|/|z|)^{\pi/\varphi}$. Plugging this in \eqref{pgeneral}, we obtain \eqref{pbisec}.
	\end{proof}
	In the course of the proofs to come, we will need to apply the previous result for specific rotations of the angular sectors $S_{z_0}(\alpha,\beta)$. Because of this, in the next result we adapt Proposition \ref{lematecnico} to this setting.
	\begin{corollary}\label{calculo medida armonica}
		Let $\psi\in[\pi/2-\beta,\pi/2+\al]$, $\al,\beta \in [0,\pi]$ with $0<\al+\beta\leq\pi$, define $R(\psi)=e^{i\psi}S(\alpha,\beta)$ and 
		$$
		E_t:=\{w\in\partial R(\psi): \ \emph{Re}(w)>t\},\quad t>1.
		$$
		Set $z_t:=\partial R(\psi)\cap\{\emph{Re}(z)=t\}$ and $\varphi:=\al+\beta$. Then, if $z=|z|e^{i\theta_0}\in R(\psi)$, 
		\[
		\omega(z,E_t,R(\psi))=\frac{1}{2\pi}\arctan\left(\frac{2T}{T^2-1}\right),
		\]
		where
		\[
		T=\cot(\kappa)\left( \frac{1}{\cos(\kappa)}\left(\frac{|z_{t}|}{|z|}\right)^{\frac{\pi}{\varphi}}
		-1
		\right),
		\]
		 and $\kappa=(\theta_0+\alpha-\psi)\frac\pi\varphi$, for all $T>1.$ In particular, if $\theta=\psi+(\alpha+\beta)/2$ and $z =re^{i\theta}$, $r>0$,
		\begin{equation}\label{lemabisec}
			\omega(z,E_t, R(\psi))        
			=
			\frac{1}{2\pi}\arctan\left(
			\frac{2(|z_{t}|r)^{\pi/\varphi}}{|z_{t}|^{2\pi/\varphi}-r^{2\pi/\varphi}}
			\right),
		\end{equation} 
	for all $r<|z_t|$.
	\end{corollary}
	\begin{proof}
	It is enough to apply Proposition \ref{lematecnico} to the sector $R(\psi)$, given \eqref{pgeneral} with $\kappa=(\theta_0+\alpha-\psi)\frac{\pi}{\varphi}\cdot$	\end{proof}
	
	We will also need to apply Proposition \ref{lematecnico} in slightly different situations regarding the boundary set considered. For the sake of simplifying the upcoming proofs, we include the following result:
	
	\begin{corollary}\label{corolario medida armonica imaginaria}
		Let $\Fi \in (0,\pi]$, consider $\widetilde{S}= e^{i\pi/2}S(0,\Fi)$ and
		$$E_t:= \{ w\in \partial \widetilde{S}: \PR(w) = 0, \ \Impart(w) > t \}, \quad t>1.$$ Then, if $z=|z|e^{i\theta_0}\in \widetilde{S}$
		$$
		\omega(z,E_t,\widetilde{S}) = \frac{1}{2\pi}\arctan\left(\frac{2T}{T^2-1}\right), \qquad \textnormal{with } T=\frac{t^{\frac{\pi}{\Fi}}-|z|^{\frac{\pi}{\Fi}}\cos(\kappa)}{|z|^{{\frac{\pi}{\Fi}}}\sin(\kappa)},
		$$
		 
			where $\kappa = (\theta_0-\frac{\pi}{2})\frac{\pi}{\Fi}$, for all $T>1$.
	\end{corollary} 
	\begin{proof}
		Considering the rotation-invariance of the harmonic measure, by applying Proposition \ref{lematecnico} after rotating an angle of $-\min\{\frac{\pi}{2},\Fi\}$, the result follows.
	\end{proof}

The computation of the harmonic measure of specific boundary subsets of vertical strips will be needed as well:
\begin{proposition}\label{LTecnico:banda vertical}
Let $t>0$. Consider the vertical strip of width $2t$ given by
\begin{equation*}
    B_t=\{z\in\C: \ |\Re(z)|<t\}.
\end{equation*}
Let $R_t$ be such that $\lim_{t\to\infty}\frac{R_t}{t}=+\infty$. Then, there exists $t_0>0$ such that for all $t>t_0$
 \begin{equation*}
     \omega(0,\{z\in\partial B_t: \ \Im(z)>R_t\}, B_t)
     \approx e^{-{\frac\pi2}\frac{R_t}{t}},
 \end{equation*}
 where the underlying constants do not depend on $t$.
\end{proposition}
\begin{proof}
     For $t>0$, set $E_t=\{z\in\partial B_t: \ \Im(z)>R_t\}$. Consider $F(z)=e^{-i\frac{\pi}{2t}z}$, $z\in B_t$. Note that $F\colon B_t\to\C_+$ conformally. It is also clear that
    \[
    F_t:=F(E_t)=\{z\in i\R: \ |z|>e^{{\frac\pi2}\frac{R_t}{t}}\}.
    \]    
    Then, by the conformal invariance of harmonic measure,
    \begin{align}\label{p1:confinv}
        \omega(0,E_t,B_t)=\omega(F(0),F(E_t),F(B_t))=
        \omega(1,F_t,\C_+).
    \end{align}
    Set $X_t:=e^{{\frac\pi2}\frac{R_t}{t}}$.  Now, by \cite[Example 7.2.5]{BCD},
    \begin{equation}\label{p2:lema+asintotico}
        \omega(1,F_t,\C_+)=\frac{1}{\pi}\arctan\left(
         \frac{2X_t}{X^2_t-1}
        \right)
        \approx\frac{1}{X_t},
    \end{equation}
    for $t>t_0$, for some $t_0$ large enough. Hence, using \eqref{p1:confinv} and \eqref{p2:lema+asintotico}, we get that
    \begin{equation*}
        \omega(0,\{z\in\partial B_t: \ \Im(z)>R_t\}, B_t)
     \approx X_t^{-1}=e^{-{\frac\pi2}\frac{R_t}{t}}.
    \end{equation*}    \end{proof}
    
\section{Main results}\label{seccion general}	

In order to exploit the connection between harmonic measure and the $H^p$-norm from Corollary \ref{corolario formula norma}, we will study Koenigs domains contained in a sector. More specifically, in this section, we introduce some geometric concepts related to the domain $\Omega$. Using these, we provide sufficient conditions for membership of $e^{\lambda h}$ to $H^p$. 

\medskip

We start by introducing the following definition:

\begin{definition}\label{definicion outer maximal angles} Let $\Omega$ be a parabolic Koenigs domain such that $0\in\Omega$. Assume that
$\Omega\subseteq S_p(\al_0,\beta_0),$ for some $p\in \R$, $\al_0,\beta_0 \in [0,\pi]$, and $0<\al_0+\beta_0\leq\pi$. We define the \textit{outer lower maximal angle} $\al_\O^o$ of $\Omega$  as
$$
\al_\O^o := \inf \{ \al\in [0,\pi]: \  \textnormal{there exists } p\in \R \textnormal{ such that } \Omega \subseteq p+e^{i\left(\frac{\pi}{2}-\al\right)}\C_+\}. 
$$
The \textit{outer upper maximal angle} $\beta_\O^o$ of $\Omega$ is given by
$$
\b_\O^o := \inf \{ \b\in [0,\pi]: \  \textnormal{there exists } p\in \R \textnormal{ such that } \Omega \subseteq p+e^{-i\left(\frac{\pi}{2}-\b\right)}\C_+\}. 
$$
\end{definition}
Note that both quantities are well defined and satisfy $\al_\O^o \in [0,\al_0)$, $\beta_\O^o\in [0,\beta_0).$ In particular, $0\leq \a_\O^o+\b_\O^o\leq\pi$. Observe that the containment $\Omega \subseteq S_p(\al_0,\beta_0)$ is equivalent to 
$$
\Omega\subseteq (p+e^{-i\left(\frac{\pi}{2}-\b_0\right)}\C_+)\cap (p+e^{i\left(\frac{\pi}{2}-\al_0\right)}\C_+).
$$
Likewise, 

\begin{definition} Let $\Omega$ be a parabolic Koenigs domain with $0\in\Omega$.  Assume that
$ S_p(\al_1,\beta_1)\subseteq \Omega,$ for some $p\in \R$, $\al_1,\beta_1 \in [0,\pi)$, $0<\al_1+\beta_1<\pi$. We define the \textit{inner lower maximal angle} $\al_\O^i$  of $\Omega$  as 
$$
\al_\O^i := \sup (\{0\}\cup \{ \al\in [0,\pi] : \textnormal{there exists } p\in \R\cap\Omega \textnormal{ such that }  S_p(\al,0)\subseteq \Omega  \}). 
$$
The \textit{inner upper maximal angle} $\b_\O^i$ of $\Omega$ is given by
$$
\b_\O^i := \sup (\{0\} \cup \{ \b\in [0,\pi] : \textnormal{there exists } p\in \R\cap\Omega \textnormal{ such that }  S_p(0,\b)\subseteq \Omega  \}). 
$$
On the other hand, if $\Omega$ does not contain any angular sector, we define $\al_\O^i=\beta_\O^i = 0$.
\end{definition}
Observe that the containment $S_p(\al_1,\beta_1)\subseteq \Omega$ is equivalent to
$S_p(\al_1,0)\cup S_p(0,\beta_1)\subseteq \Omega.$ Again, $\al_\O^i$ and $\beta_\O^i$ are well defined. Moreover, under the assumption of Definition \ref{definicion outer maximal angles}, they satisfy $\al_\O^i \in [0,\al_\O^o],$ $\b_\O^i \in [0,\b_\O^o].$ In particular, $0\leq \al_\O^i+\b_\O^i\leq \pi$. 

\medskip

It is not difficult to construct examples of Koenigs domains $\Omega$ with prescribed maximal angles $\a_\O^i,\al_\O^o,\beta_\O^i,\beta_\O^o$:

\begin{example}
   Consider $0<\al_\O^i < \al_\O^o < \pi$ and $0<\beta_\O^i<\beta_\O^o < \pi$ with $\al_\O^o+\beta_\O^o < \pi$. First, we construct $\Omega'$ as the starlike at infinity domain bounded by the curve $\Gamma$, which is defined as
    $$\Gamma = V_0\cup \bigcup_{n=1}^\infty (H^\al_n\cup H^\beta_n \cup V^\al_n\cup V^\beta_n),$$
    where $V_0$ is a vertical segment, $(H^\al_n)$ and $(H^\b_n)$ are sequences of horizontal segments and $(V^\al_n)$ and $(V_n^\beta)$ are sequences of vertical segments.  We construct them in the following way:
    \begin{enumerate}
    \item [Step I.] Set $S^o = S(\al_\O^o,\beta_\O^o)$ and $S^i = S(\al_\O^i, \beta_\O^i)$. Recall that $\partial S^o = \ell(-\al_\O^o)\cup \ell(\beta_\O^o)$ and $\partial S^i = \ell(-\al_\O^i)\cup \ell(\beta_\O^i)$.
    \item [Step II.] Choose $V_0=[h_1^\al, h_1^\b]$ to be a vertical segment such that $h_1^\al \in \ell(-\al_\O^o)$ and $h_1^\b\in \ell(\b_\O^o)$.
    \item [Step III.] Define $H_1^\al = [h_1^\al, v_1^\al]$ the horizontal segment such that $v_1^\al \in \ell(-\al_\O^i)$ and $H_1^\beta =[h_1^\beta, v_1^\beta]$ such that $v_1^\b \in \ell(\b_\O^i)$
    \item [Step IV.] For each $n\in \N$, we define inductively $V_n^\al=[v_n^\al, h_{n+1}^\al]$ as the vertical segment such that $h_{n+1}^\al \in \ell(-\al_\O^o)$ and $V_n^\beta = [v_n^\b, h_{n+1}^\b]$ as the vertical segment such that $h_{n+1}^\b \in \ell(\b_\O^o)$. Likewise, define $H_{n+1}^\al = [h_{n+1}^\al, v_{n+1}^\al]$ as the horizontal segment such that $v_{n+1}^\al\in \ell(-\al_\O^i)$ and $H_{n+1}^\b = [h_{n+1}^\b, v_{n+1}^\b]$ as the horizontal segment such that $v_{n+1}^\b\in \ell(\b_\O^i)$.
    \end{enumerate}
    Finally, consider $\Omega$ as any translation of $\Omega'$ containing $0$.   By construction, $\Omega$ has the prescribed maximal angles. The same ideas, with the suitable modifications, provide examples of domains $\Omega$ with the remaining configurations of maximal angles not considered in the previous construction.
\end{example}

The inner and outer maximal angles serve us as a starting point to characterize the point spectrum of $\Delta$. We illustrate this fact through the following containment relations:

\begin{theorem}\label{teorema contenciones general}
    Let $(\Fi_t)_{t>0}$ be a parabolic semigroup on $\D$. Assume that $\Omega \subseteq S_p(\al_0,\beta_0)$ for some $p\in \R$, $\al_0,\beta_0\in [0,\pi]$ with $0<\al_0+\beta_0\leq\pi$, and let $\al_\O^o$ and $\beta_\O^o$ be the outer maximal angles of $\Omega$. Then
    $$
    \esp \supseteq \left\{ re^{i\theta}: r\geq 0, \frac{\pi}{2}+\al_\O^o< \theta < \frac{3\pi}{2}-\beta_\O^o\right\}.
    $$
    Moreover, if $\Omega$  has inner maximal angles $\al_\O^i$ and $\beta_\O^i$, then
    $$
    \esp \subseteq  \left\{ re^{i\theta}: r\geq 0, \frac{\pi}{2}+\al_\O^i\leq \theta \leq \frac{3\pi}{2}-\beta_\O^i\right\}.
    $$
\end{theorem}
\begin{proof}
    The result follows as a consequence of the definition of the outer and inner maximal angles of $\Omega$ and Proposition  \ref{contencion espectro puntual}.
\end{proof}\newpage
	\begin{figure}[h!!]
	\centering
	\resizebox{0.60\textwidth}{!}{
		\begin{tikzpicture}[>=stealth]
			
			% Ejes
			\draw[->] (-5,0) -- (5,0) node[right] {$x$};
			\draw[->] (0,-5) -- (0,5) node[above] {$y$};
			
			% Origen
			\node at (0,0) [below left] {$0$};
			
			% Región sombreada entre las dos semirrectas
			\fill[blue!20, opacity=0.6] 
			(0,0) -- ({5*cos(150)}, {5*sin(150)}) 
			arc (150:210:5) 
			-- cycle;

            % Región sombreada entre las dos semirrectas exteriores
			\fill[blue!10, opacity=0.4] 
			(0,0) -- ({5*cos(120)}, {5*sin(120)}) 
			arc (120:240:5) 
			-- cycle;
			
			% Semirrecta superior (2pi/3), ahora completamente sólida
			\draw[->, dashed] (0,0) -- ({5*cos(150)}, {5*sin(150)});

            % Otra semirrecta
			\draw[->, thick] (0,0) -- ({5*cos(120)}, {5*sin(120)});
			
			% Semirrecta inferior (4pi/3), ahora completamente sólida
			\draw[->, dashed] (0,0) -- ({5*cos(210)}, {5*sin(210)});

            % Otra semirrecta
			\draw[->, thick] (0,0) -- ({5*cos(240)}, {5*sin(240)});
			
			% Etiqueta de la región
			\node at ({2.8*cos(180)}, {2.8*sin(180) + 1}) {$\sigma_p(\Delta_{\mid H^p})$};
			
			% Ángulo alpha_Ω: desde eje y positivo (90°) hasta 150°
			\draw (0,0) ++(90:1.2) arc (90:150:1.2);
			\node at ({1.1*cos(120)}, {1.1*sin(120)+0.5}) {$\alpha^o_\Omega$};

            % Otro ángulo
			\draw (0,0) ++(90:2) arc (90:120:2);
			\node at ({2*cos(120)}, {2*sin(120)+0.5}) {$\alpha^i_\Omega$};
			
			% Ángulo beta_Ω: desde eje y negativo (270°) hasta 210°
			\draw (0,0) ++(270:1.2) arc (270:210:1.2);
			\node at ({1.1*cos(240)}, {1.1*sin(240)-0.5}) {$\beta^o_\Omega$};

            % Otro ángulo
			\draw (0,0) ++(270:2) arc (270:240:2);
			\node at ({2*cos(240)}, {2*sin(240)-0.5}) {$\beta^i_\Omega$};
	\end{tikzpicture}}
    \caption{Containments of $\esp$ in Theorem \ref{teorema contenciones general}.}
\end{figure}

The next result allows us to improve the containment relations for $\esp$ from Theorem \ref{teorema contenciones general}:

  \begin{proposition}\label{proposicion suma espectros}
       Let $(\Fi_t)_{t>0}$ be a parabolic semigroup on $\D$.  Assume that $\Omega \subseteq S_p(\al_0,\beta_0)$ for some $p\in \R$, $\al_0, \beta_0\in [0,\pi]$, with $\alpha_0+\beta_0\leq\pi$, and let $\al_\O^o, \beta_\O^o,\al_\O^i,\beta_\O^i$ be the outer and inner maximal angles of $\Omega$.
       \begin{enumerate}
           \item [(i)] Let $\lambda_0 = |\lambda_0|e^{i\left(\frac{\pi}{2}+\sigma\right)} \in \esp$,  with $\sigma \in [\al_\O^i,\al_\O^o]$ and $|\lambda_0| >0.$ Then
           \begin{equation}\label{contencion suma}
           \left\{re^{i\theta}: r\geq 0, \frac{\pi}{2}+\al_\O^o < \theta < \frac{3\pi}{2}-\beta_\O^o\right\} +\left\{re^{i\left( \frac{\pi}{2}+\sigma\right)}: 0\leq r\leq |\lambda_0|\right\}\subseteq \esp. 
           \end{equation}
           Moreover, if $|\lambda_0|e^{i\left( \frac{\pi}{2}+\sigma\right)} \in \esp$ for all $|\lambda_0|>0$, then
           \begin{equation}\label{contencion suma 2}
               \left\{re^{i\theta}: r\geq 0, \frac{\pi}{2}+\sigma \leq \theta < \frac{3\pi}{2}- \beta_\O^o\right\}\subseteq \esp.
           \end{equation}        
           \item [(ii)] Let $\lambda_0 = |\lambda_0|e^{i\left(\frac{3\pi}{2}-\nu\right)} \in \esp$,  with $\nu \in [\b_\O^i,\b_\O^o]$ and $|\lambda_0| >0.$ Then
           \begin{equation*}
               \left\{re^{i\theta}: r\geq 0, \frac{\pi}{2}+\al_\O^o < \theta < \frac{3\pi}{2}-\beta_\O^o\right\} +\left\{re^{i\left( \frac{3\pi}{2}-\nu\right)}: 0\leq r\leq |\lambda_0|\right\}\subseteq \esp.
           \end{equation*} 
           Moreover, if $|\lambda_0|e^{i\left( \frac{3\pi}{2}-\nu\right)} \in \esp$ for all $|\lambda_0|>0$, then
           \begin{equation*}
               \left\{re^{i\theta}: r\geq 0, \frac{\pi}{2}+\al_\O^o < \theta \leq \frac{3\pi}{2}- \nu \right\}\subseteq \esp.
           \end{equation*}
       \end{enumerate}
  \end{proposition}
  \begin{proof}
     It suffices to prove $(i)$, the proof of $(ii)$ being identical. First, observe that, by the definition of the outer maximal angles, for every $\mu \in \left\{re^{i\theta}: r\geq 0, \frac{\pi}{2}+\al_\O^o < \theta < \frac{3\pi}{2}-\beta_\O^o\right\}$, we have that $\mu \Omega$ is contained in a left half-plane. Hence, by Proposition \ref{proposicion contenido semiplano}, $e^{\mu h}\in H^\infty$. Now, if $\lambda_0 \in \esp,$ according to Proposition \ref{proposicion convexidad}, $\lambda = |\lambda|e^{i\left(\frac{\pi}{2}+\sigma\right)} \in \esp$ for all $|\lambda|\leq |\lambda_0|$. Now, let $\mu  \in \left\{re^{i\theta}: r\geq 0, \frac{\pi}{2}+\al_\O^o < \theta < \frac{3\pi}{2}-\beta_\O^o\right\}$ and $\lambda$ be as above. The result follows from applying Proposition \ref{proposicion contenido semiplano}.
  \end{proof}
  
Based on the latter results, in order to completely characterize the point spectra of infinitesimal generators associated to parabolic Koenigs domains $\Omega$, it remains to study whether $e^{\lambda h} \in H^p$ for those $\lambda$ not covered by Theorem \ref{teorema contenciones general}. That is, determine whether $\lambda = |\lambda|e^{i\left(\frac{\pi}{2}+\sigma\right)}$, with $\sigma \in \left[\al_\O^i,\al_\O^o\right]$ or $\mu = |\mu|e^{i\left(\frac{3\pi}{2}-\nu\right)}$, with $\nu \in \left[\b_\O^i,\b_\O^o\right]$ belong to $\esp$ or not. To do so, in view of Corollary \ref{corolario formula norma}, it will be crucial to estimate the $H^p$-norm in terms of harmonic measure.

\medskip
To find such estimates, we need a manageable description of the boundary of the rotated domains $\Omega_\sigma = e^{i\sigma}\Omega$ and $\Omega_{\nu} = e^{-i\nu}\Omega$. Since these are not Koenigs domains, they do not have an \textit{a priori} associated defining function, in the sense of the one introduced in page \pageref{pagina defining functions}. However, we may regularize the rotated Koenigs domains via the following concept, giving proper defining functions:

\begin{definition}\label{definicion hull}
    The \textit{starlike at infinity hull} of a domain $G\subset \C$, denoted by $\coinf(G)$, is defined as

$$
\coinf(G) = \bigcup_{t\geq 0} (G+t).
$$
\end{definition}
Observe that $\coinf(G)$ is the smallest starlike at infinity domain containing $G$. If $\coinf(G)\neq \C$, then $\coinf(G)$ is a Koenigs domain.

\medskip

Now, let $\Omega$ be a parabolic Koenigs domain, $0\in \Omega$, such that $\Omega \subseteq S_p(\al_0,\beta_0)$ for some $p\in \R$, $\al_0\in [0,\pi),\beta_0\in [0,\pi)$. Also, let $\al_\O^o, \beta_\O^o,\al_\O^i,\beta_\O^i$ be the outer and inner maximal angles of $\Omega$, respectively. For each $\sigma \in [\al_\O^i,\al_\O^o]$ and $\nu \in [\beta_\O^i, \beta_\O^o]$ consider $\Omega_\sigma = e^{i\sigma}\Omega$ and $\Omega_\nu = e^{-i\nu}\Omega$. Then, $\coinf(\Omega_\sigma)$ and $\coinf(\Omega_\nu)$ are parabolic Koenigs domains with defining functions 
$\psi_{\Omega_\sigma}:I_\sigma\rightarrow [-\infty,+\infty)$ and $\psi_{\Omega_\nu}:I_\nu\rightarrow [-\infty,+\infty)$, respectively.

\begin{remark}
The defining functions $\psi_{\Omega_\s}$ and $\psi_{\Omega_\nu}$ introduced above can be geometrically described in terms of $\Omega$. We discuss the construction of $\psi_{\Omega_\s}$, the one for $\psi_{\Omega_\nu}$ being identical. Consider $\Omega_\s = e^{i\sigma}\Omega$ and define $I_\s=\{\Impart(z)\colon z\in\Omega_\s\}.$ Now, for each $y \in I_\s$, consider the line $L_y = \{z\in \C\colon  \Im(z)=y\}.$ It is clear that $L_y$ intersects $\Omega_\sigma$. Consider now $\kappa_{\Omega}(y)= \inf \{\PR(z)\colon z\in L_y\cap \Omega_\sigma\}$. It is straightforward to check that $\kappa_{\Omega}$ coincides with the defining function $\psi_{\Omega_\sigma}$ of $\coinf(\Omega_\sigma)$.
\end{remark}

The following two results provide some basic properties of the defining functions that shall be needed in what follows.

\begin{proposition}\label{proposicion no convexo defining functions}
     Let $(\Fi_t)_{t>0}$ be a parabolic semigroup on $\D$.  Assume that $\Omega \subseteq S_p(\al_0,\beta_0)$ for some $p\in \R$, $\al_0\in [0,\pi),\beta_0\in [0,\pi)$ with $0<\al_0+\beta_0<\pi$. Suppose that $\al_\O^o+\beta_\O^o > 0$.
         \begin{enumerate}

             \item [(a)] Let $\sigma \in [\al_\O^i,\al_\O^o]$ and consider $\psi_{\Omega_\sigma}: I_\sigma \rightarrow [-\infty,+\infty).$ There exists $a \in \R$ such that $[a,+\infty) \subset I_\s$. Moreover, $I_\s = \R$ if and only if either
             \begin{itemize}

\smallskip
             
                 \item[i)] $\sigma< \al_\O^o$, or

\smallskip
                 
                 \item[ii)] $\sigma=\al_\O^o$ and $\Omega$ is not contained in any sector $S_{z_0}(\al_\Omega^o,\beta)$,  with $z_0\in \R$ and $\beta \in [\beta_\Omega^i,\pi)$.
             \end{itemize}
              In such a case, there exist $t_0>0$ and $\delta>0$  such that $\psi_{\Omega_{\s}}(-t) \geq \delta$ for all $t\in [t_0,+\infty)$ and $$\limsup\limits_{t\rightarrow + \infty} \psi_{\Omega_\sigma}(-t) = +\infty.$$
             \item [(b)] Let $\nu \in [\b_\O^i,\b_\O^o]$ and consider $\psi_{\Omega_\nu}: I_\nu \rightarrow [-\infty,+\infty).$ There exists $b \in \R$ such that $(-\infty,b] \subset I_\nu$. Moreover, $I_\nu = \R$ if and only if, either
             \begin{itemize}

             \smallskip
             
                 \item[i)] $\nu< \b_\O^o$ or,

\smallskip
                 
                 \item[ii)] $\nu=\b_\O^o$ and $\Omega$ is not contained in any sector $S_{z_0}(\al,\beta_\O^o)$,  with $z_0\in \R$ and $\al \in [\al_\Omega^i,\pi)$.
             \end{itemize}
              In such a case, there exist $t_0>0$ and $\delta >0$ such that $\psi_{\Omega_{\nu}}(t) \geq \delta$ for all $t\in [t_0,+\infty)$ and $$\limsup\limits_{t\rightarrow+\infty} \psi_{\Omega_{\nu}}(t) = + \infty.$$
         \end{enumerate}
         \end{proposition}
\begin{proof}
 We show $(a)$, the proof for $(b)$ is analogous. First, let us prove $(i)$.

 \medskip
Observe that for each $\al < \al_\O^i$ and $\beta<\beta_\O^i$, $\Omega_\sigma$ contains the rotated angular sector $e^{i\sigma}S_z(\al,\beta)$, for some $z\in \C$. In particular, it shows that $\PR(\O_\s)=\{\PR(z)\colon z\in\Omega_\s\}$ is not bounded above, so there exists $a\in \R$ such that $[a,+\infty)\subset I_\sigma$. Now, $I_\sigma \neq \R$ if and only if $\PR(\Omega_\sigma)$ is bounded below, that is, if and only if there exists $y_0 \in \R$ such that $\Omega_\sigma\subset y_0 + i\C_+$. This is equivalent to asking $\Omega \subset z+e^{i\left(\frac{\pi}{2}-\sigma\right)}\C_+$ for some $z\in \C_+$. The definition of $\al_\O^o$ forces $\sigma = \al_\O^o$, and then $I_\sigma \neq \R$ if and only if $\Omega \subset z+e^{i\left(\frac{\pi}{2}-\al_\Omega^o\right)}\C_+$ for some $z\in \C_+$, which is again equivalent to $\Omega \subset S_z(\al_\Omega^o,\beta)$ for $\beta \in [\beta_\O^i,\pi)$ and for some $z\in \C_+$, as claimed.

Now, assume $I_\sigma = \R$. Observe that, there exists $z\in\C$, such that $\Omega_\sigma$ is contained in a rotated sector $e^{i\sigma} S_z(\al,\beta)$ for each $\al>\al_\O^o$ and $\beta>\beta_\O^o$. As a consequence,  for some $t_0>0$ and $\delta >0$ the function $\psi_{\O_\s}$ is greater than $\delta$ in the interval $[t_0,+\infty).$ Moreover, the previous containment also provides that $ \limsup\limits_{t\rightarrow + \infty} \psi_{\Omega_{\sigma}}(-t) = + \infty.$
\end{proof}

In the case where $\al_\O^o+\beta_\O^o = 0$, by Theorem \ref{teorema contenciones general}, it only remains to study the points $\lambda = |\lambda|e^{i\frac{\pi}{2}}$ and $\mu = |\mu|e^{i\frac{3\pi}{2}}$. To this end, it will suffice to consider the defining function of $\Omega$, namely, $\psi_{\Omega}$:
\begin{proposition}\label{proposicion no convexo defining functions no contiene sector}
     Let $(\Fi_t)_{t>0}$ be a parabolic semigroup on $\D$.  Assume that $\Omega \subseteq S_p(\al_0,\beta_0)$ for some $p\in \R$, $\al_0\in [0,\pi),\beta_0\in [0,\pi)$ with $0<\al_0+\beta_0<\pi$. Assume that $\al_\O^o+\beta_\O^o = 0$. Consider $\psi_\O\colon I \rightarrow [-\infty,\infty)$ the defining function of $\Omega.$ Then, 
     \begin{enumerate}
         \item [(a)] $I=\R$ if and only if either
         \begin{itemize}
             \item[i)]$\Omega$ is not contained in any sector $S_{z_0}(\al,\beta)$ with $\al=0$ and $\beta \in (0,\pi)$ or,
             \item[ii)]$\beta=0$ and $\al \in (0,\pi)$.
         \end{itemize} In such a case, there exist $t_0$ and $\delta>0$  such that $\psi_{\Omega}(t) \geq \delta$ for all $|t|\geq t_0$ and 
         $$
         \lim\limits_{t\rightarrow + \infty} \frac{\psi_{\Omega}(-t)}t = \lim\limits_{t\rightarrow + \infty} \frac{\psi_{\Omega}(t)}t=+\infty.
         $$
	\item [(b)] $I= (a,+\infty)$ if and only if $\Omega$ is contained in a sector $S_{z_0}(0,\beta)$ for some $z_0\in\R$ and $\beta \in (0,\pi)$. In such a case, there exist $t_0, \delta  >0$ such that $\psi_\O(t) \geq \delta$  for all $t\geq t_0$ and 
    $$
    \lim\limits_{t\rightarrow+\infty} \frac{\psi_\O(t)}t = + \infty.
    $$
\item [(c)] $I= (-\infty,b)$ if and only if $\Omega$ is contained in a sector $S_{z_0}(\al,0)$ for some $z_0\in\R$ and $\al \in (0,\pi)$. In such a case, there exists $t_0, \delta  >0$ such that $\psi_\O(-t) \geq \delta$  for all $t\leq -t_0$ and 
$$
\lim\limits_{t\rightarrow+\infty} \frac{\psi_\O(-t)}t = + \infty.
$$
     \end{enumerate}
\end{proposition}
\begin{proof}
    Note that, since $\Omega$ is parabolic, $I$ must be connected and unbounded. This yields the three stated situations. 
    
    For the proof of the limits, we show that, under the assumption $I=\R$, $$\lim\limits_{t\rightarrow + \infty} \frac{\psi_{\Omega}(t)}t=+\infty.$$ The proof of the other cases is analogous. Assume to the contrary that there exists a sequence $(t_n)_{n\in \N}$ with $t_n\rightarrow+\infty$, $n\rightarrow+\infty$ and such that
    $\frac{\psi_\O(t_n)}{t_n}\leq C $ for some $C>0$. In particular, $\psi_\O(t_n) \leq Ct_n$, which implies that the sequence $\psi_\O(t_n)+it_n$ does not belong to the angular sector $S(\arctan(C))$. However, since $\al_\O^o +\beta_\O^o = 0$, $\Omega$ is contained in all sectors $S_{p_\varepsilon}(\varepsilon)$ with $\varepsilon>0$ and $p_\varepsilon$ a suitable point in $\R$. This leads to a  contradiction.
    
    The rest of the proof follows the same ideas as the proof of Proposition \ref{proposicion no convexo defining functions} with the obvious modifications.
\end{proof}

From this point, we will distinguish between the cases in which $\al_\O^o+\beta_\O^o >0$ and  $\al_\O^o+\beta_\O^o =0$.

\medskip
%\subsection{Case I: \texorpdfstring{$\bm{\alpha_\Omega^o + \beta_\Omega^o > 0}$}{c}} 
We begin with the case $\al_\O^o+\beta_\O^o >0$. In the next result, through a careful estimate of the harmonic measure, we obtain sufficient conditions to ensure the membership of $e^{\lambda h}$ to $H^p$:

\begin{theorem}\label{teorema no convexo por arriba}
     Let $(\Fi_t)_{t\geq0}$ be a parabolic semigroup on $\D$. Suppose that $\Omega \subseteq S_p(\al_0,\beta_0)$ for some $p\in \R$, $\al_0\in [0,\pi),\beta_0\in [0,\pi)$, with $\alpha_0+\beta_0\leq\pi$. Assume further that $\al_\O^o+\beta_\O^o >0$.
     \begin{enumerate}
         \item [(i)] Let $\sigma \in [\al_\O^i,\al_\O^o]$ and $\lambda =|\lambda|e^{i\left(\frac{\pi}{2}+\sigma\right)}$.  Then, $\lambda \in \esp$ if one of the following holds:
         \begin{enumerate}
             \item $\dom(\psi_{\Omega_\s})\neq \R$.
             \item There exists $\varepsilon>0$ such that
\begin{equation}\label{integral por arriba}
     \int_{t_0}^\infty e^{p|\lambda|t}\left(\frac{t}{\psi_{\O_\sigma}(-t)}\right)^{\frac{\pi}{\sigma+\beta_\O^o+\varepsilon}}dt < \infty.
\end{equation}             
         \end{enumerate}
         \item [(ii)]  Let $\nu \in [\b_\O^i,\b_\O^o]$ and $\lambda =|\lambda|e^{i\left(\frac{3\pi}{2}-\nu\right)}$.  Then, $\lambda \in \esp$ if one of the following holds:
         \begin{enumerate}
             \item $\dom(\psi_{\Omega_\nu})\neq \R$.
             \item There exists $\varepsilon>0$ such that 
$$
    \int_{t_0}^\infty e^{p|\lambda|t}\left(\frac{t}{\psi_{\O_\nu}(t)}\right)^{\frac{\pi}{\al_\O^o+\nu+\varepsilon}}dt < \infty.
    $$            
         \end{enumerate}
      
     \end{enumerate}
\end{theorem}
\begin{remark}
		Recall that the real number $t_0$ appearing  in the previous result is the one provided in Proposition \ref{proposicion no convexo defining functions}. By the properties stated therein, observe that $\psi_{\Omega_\s}(-t)$ and $\psi_{\Omega_\nu}(t)$ are bounded below by some $\delta >0$ for all $t\geq t_0$, so the convergence behaviour of the integrals depends only on the growth of the functions $\psi_{\Omega_\s}(-t)$ or $\psi_{\Omega_\nu}(t)$ as $t\rightarrow+\infty$.
	\end{remark}
\noindent \textit{Proof of Theorem \ref{teorema no convexo por arriba}}.  We show $(i)$, the proof for $(ii)$ is analogous. 
    
    If $\dom(\psi_{\Omega_\s}) \neq \R$, then by Proposition \ref{proposicion no convexo defining functions}, $\s = \al_\O^o$ and $\Omega$ is contained in a sector $S_p(\al_\O^o,\beta)$ for some $p\in \R$ and $\beta \in [\b_\O^i,\pi).$ Then, $\lambda\Omega$ is contained in a left half-plane and Proposition \ref{proposicion contenido semiplano} yields that $\lambda \in \esp$, as desired.
    
    Assume now that $\dom(\psi_{\Omega_\s})=\R$. First, recall that, by hypothesis, we can apply Corollary \ref{corolario formula norma} to obtain
    $$
    \norm{e^{\lambda h}}_{H^p} = \frac{p|\lambda|}{2\pi} \int_{-\infty}^\infty e^{p|\lambda|t}\omega(0, E_t,\Omega_\theta)dt,
    $$
    where $E_t = \{\xi \in \partial \Omega_\theta: \PR(\xi) >t\}$ and $\theta = \frac{\pi}{2}+\sigma$. 
    
    Let $\varepsilon \in (0,\frac\pi2)$ be chosen such that $0<\sigma+\beta_\O^o+\varepsilon<\pi$ and take  $t\geq t_0$, for some $t_0$ large enough. Observe that $$\inf \{\Impart(\xi) : \xi \in \Omega_\theta, \ \PR(\xi) = t\} = \psi_{\Omega_\sigma}(-t).$$    By the definition of $\beta_\O^o$, there exists $p \in \C$ such that $\ell_p(\beta_\O^o+\varepsilon)\cap \Omega = \varnothing$ (see \eqref{notacionlados}), which implies that $\ell_{-iy_0}
    (\frac{\pi}{2}+\sigma+\beta_\O^o+\varepsilon)\cap \Omega_\theta = \varnothing$ for some $y_0>0$. Now, consider the set
    $$
    S= e^{i\frac{\pi}{2}}S(0,\sigma+\beta_\O^o+\varepsilon).
    $$
    Define also $S_t := \z_t +S$, where
    $$\z_t = t-i(y_0+t\tan(\tilde{\theta})) \qquad (t\geq t_0),$$ with
    \begin{equation*}
		\tilde{\theta} = \left\{\begin{matrix}
			\sigma+\beta^o_\Omega + \varepsilon & \textnormal{if } \sigma+\beta^o_\Omega + \varepsilon < \pi/2, \\
			\varepsilon & \textnormal{if } \sigma+\beta_\Omega^o + \varepsilon \geq \pi/2.
		\end{matrix}\right.
	\end{equation*}
    Observe that $\partial S_t = \ell_{\z_t}\left(\frac{\pi}{2}\right)\cup \ell_{\z_t}\left(\frac{\pi}{2}+\sigma+\beta_\O^o+\varepsilon\right)$. By construction, the second half-line does not intersect $\Omega_\theta$, while the first one may intersect in infinitely many points. Notice also that the infimum of the imaginary parts of these intersection points is precisely $\psi_{\Omega_\sigma}(-t)$. 

    Now, define $F_t= \{ t+iy\in  \partial S_t : y > \psi_{\Omega_\sigma}(-t)\}$ for all $t\geq t_0$ and observe that $\partial(S_t\cap \Omega_\theta) \subseteq F_t\cup (\partial \Omega_\theta \setminus E_t).$ As a consequence, applying the Maximum Principle, we deduce that
    \begin{equation}\label{pmax}
         \omega(0,E_t,\Omega_\theta) \leq \omega(0,F_t,S_t) 
    \end{equation}
    for all $t\geq t_0$. By the translation invariance of the harmonic measure, we have that
    $$
    \omega(0,F_t,S_t) = \omega(-\z_t,F_t-\z_t,S),
    $$
    where $F_t-\zeta_t = \{z\in \C: \PR(z) = 0,\Im(z) > \psi_{\Omega_\s}(-t)+y_0+\tan(\tilde{\theta})\}.$ Then, applying Corollary \ref{corolario medida armonica imaginaria},  we have that, for $t\geq t_0$,

    $$
     \omega(0,F_t,S_t) = \frac{1}{2\pi}\arctan\left(\frac{2T_t}{T_t^2-1}\right)$$
     for all $T_t > 1$, where 
 $$
	 T_t = \frac{(\psi_{\O_\s}(-t)+y_0+\tan(\tilde{\theta}))^{\frac{\pi}{\sigma+\b_\O^o+\varepsilon}}-|\zeta_t|^{\frac{\pi}{\sigma+\b_\O^o+\varepsilon}}\cos(\kappa_t)}{|\zeta_t|^{\frac{\pi}{\sigma+\b_\O^o+\varepsilon}}\sin(\kappa_t)},
	 $$
	 and $\kappa_t = (\arg(-\zeta_t)-\frac{\pi}{2})\frac{\pi}{\sigma+\b_\O^o+\varepsilon}$. Recall that $|\zeta_t|=t\sqrt{1+(y_0/t+\tan(\tilde{\theta}))^2}$ and that $$\kappa_t \rightarrow (\frac{\pi}{2}-\tilde{\theta})\frac{\pi}{\sigma+\b_\O^o+\varepsilon},\quad \text{as }t\rightarrow + \infty.$$ As a consequence, for all $t\geq t_1$
     $$T_t \gtrsim \left(\frac{\psi_{\O_\s}(-t)}{t} \right)^{\frac{\pi}{\al_\Omega+\beta_\Omega+\varepsilon}},$$ where $t_1\geq t_0$. Moreover, there exists $c>0$ such that  $T_t >1$ for all $t\geq t_1$ satisfying $\psi_{\Omega_\sigma}(-t) \geq ct$. Thus, denoting
     $$
     A=\{ x \in [t_1,+\infty) : \psi_{\Omega_\s}(-x) \geq cx\}\subset [t_1,+\infty),
     $$
     we have
  \begin{equation*}
      \begin{split}
          & \int_A  e^{p|\lambda|t}\left(\frac{t}{\psi_{\O_\sigma}(-t)}\right)^{\frac{\pi}{\sigma+\beta_\O^o+\varepsilon}}dt  + \int_{[t_1,+\infty)\setminus A} e^{p|\lambda|t}\left(\frac{t}{\psi_{\O_\sigma}(-t)}\right)^{\frac{\pi}{\sigma+\beta_\O^o+\varepsilon}}dt
          \\ &= \int_{t_1}^\infty e^{p|\lambda|t}\left(\frac{t}{\psi_{\O_\sigma}(-t)}\right)^{\frac{\pi}{\sigma+\beta_\O^o+\varepsilon}}dt < \infty.
      \end{split}
  \end{equation*}
  Since the second integral on $[t_1,+\infty)\setminus A$ is finite, we have that
  $$
  +\infty > \int_{[t_1,+\infty)\setminus A} e^{p|\lambda|t}\left(\frac{t}{\psi_{\O_\sigma}(-t)}\right)^{\frac{\pi}{\sigma+\beta_\O^o+\varepsilon}} \geq \frac{1}{c^\frac{\pi}{\sigma+\beta_\O^o+\varepsilon}}\int_{[t_1,+\infty]\setminus A}e^{p|\lambda| t} dt.
  $$
Recalling that the harmonic measure is a probability measure and using \eqref{pmax}, there exists $C>0$ such that
$$
\norm{e^{\lambda h}} \leq C + \int_A e^{p|\lambda|t} \omega(0,F_t,S_t)dt+ \int_{[t_1,+\infty)\setminus A} e^{p|\lambda|t}dt.
$$
Hence, $\norm{e^{\lambda h}}_{H^p} < \infty$ if  
\begin{equation}\label{integral caso no convexo}
    \int_A e^{p|\lambda|t} \omega(0,F_t,S_t)dt<+\infty.
\end{equation}
Therefore, it remains to show that the previous integral converges. Recall that $T_t >1$ for all $t\in A$, so we have that
$$\omega(0,F_t,S_t) = \frac{1}{2\pi}\arctan\left(\frac{2T_t}{T_t^2-1}\right).$$ Noticing that $\arctan(x)\leq x$ for all $x\in \R$, we have for all $t\geq t_1$, 
$$
\omega(0,F_t,S_t) \lesssim \frac{1}{T_t} \lesssim \left(\frac{t}{\psi_{\O_\sigma}(-t)}\right)^{\frac{\pi}{\sigma+\beta_\O^o+\varepsilon}}.
$$
Using this in \eqref{integral caso no convexo} and then applying the hypothesis, we conclude that  the integral in \eqref{integral caso no convexo} is indeed finite. This concludes the proof.  \hfill $\Box$

\medskip

The next result establishes some geometric conditions on $\Omega$ for which it is possible to completely characterize the point spectrum of the infinitesimal generator $\Delta$. In order to ease the reading, set $\psi_{\al_\O^i} := \psi_{\O_{\al_\O^i}}$ and $\psi_{\b_\O^i} := \psi_{\O_{\b_\O^i}}$
\begin{corollary}\label{corolario no convexo contiene sector}
      Let $(\Fi_t)_{t>0}$ be a parabolic semigroup on $\D$. Suppose that $\Omega \subseteq S_p(\al_0,\beta_0)$ for some $p\in \R$, $\al_0\in [0,\pi)$, and $\beta_0\in [0,\pi)$, with $\alpha_0+\beta_0\leq\pi$.  Assume that $\al_\O^o+\beta_\O^o > 0$ and
              \begin{equation}\label{integral sin epsilon}
                   \int_{t_0}^\infty e^{p|\lambda|t} \left( \frac1{\psi_{\al_\O^i}(-t)}\right)^{\frac{\pi}{\al_\O^i+\beta_\O^o}} dt < \infty \quad \textnormal{and} \quad  \int_{t_0}^\infty e^{p|\lambda|t}\left(\frac1{\psi_{\b_\O^i}(t)}\right)^{\frac{\pi}{\al_\O^o+\beta_\O^i}} dt < \infty
              \end{equation}
    for all $|\lambda|>0$. Then,
          $$\esp = \left\{re^{i\theta}:r\geq 0, \frac{\pi}{2}+\al_\O^i \leq \theta \leq \frac{3\pi}{2}-\beta_\O^i  \right\}.$$
\end{corollary}
\begin{proof}
    First, assume that $\lambda =|\lambda|e^{i\left( \frac\pi2+\al_\O^i\right)}$ and $\mu = |\mu|e^{i\left( \frac{3\pi}2-\b_\O^i\right)}$ belong to $\esp$ for all $|\lambda|,|\mu|>0$. In such a case, Theorem \ref{teorema contenciones general} and Proposition \ref{proposicion suma espectros} give the result.

    Thus, it remains to show that $\lambda =|\lambda|e^{i\left( \frac\pi2+\al_\O^i\right)}\in \esp$; the proof for $\mu$ being identical. Let $|\lambda_0| >0$ such that
    $$
     \int_{t_0}^\infty e^{p|\lambda_0|t} \left( \frac1{\psi_{\al_\O^i}(-t)}\right)^{\frac{\pi}{\al_\O^i+\beta_\O^o}}\ < \infty.
    $$
    We show that $\lambda \in \esp$, provided that $|\lambda| < |\lambda_0|$, which together with \eqref{integral sin epsilon} will yield the result. By  Theorem \ref{teorema no convexo por arriba}, it is enough to find $\varepsilon>0$ such that 

\begin{equation}\label{integral convergente}
    \int_{t_0}^\infty e^{p|\lambda|t}\left(\frac{t}{\psi_{\al_\O^i}(-t)}\right)^{\frac{\pi}{\al_\O^i+\beta_\O^o+\varepsilon}}dt < \infty.
\end{equation}
    
    Write $|\lambda| =|\lambda_0|-\delta$, $\delta>0$. Set $\Fi = \al_\O^i+\beta_\O^o$, and define
    $$A:=\{t\in [t_0,+\infty) : \psi_{\al_\O^i}(-t) \geq e^{\frac{p|\lambda_0|\Fi}{\pi}t}\}; $$
    $$B:=  \{t\in [t_0,+\infty) : \psi_{\al_\O^i}(-t) < e^{\frac{p|\lambda_0|\Fi}{\pi}t}\}. $$ Observe that $[t_0,+\infty) =A\sqcup B.$ By the definition of the set $A$, we have
    \begin{equation*}
          \begin{split}
              \int_A e^{p|\lambda|t}\frac{t^{\frac{\pi}{\Fi+\varepsilon}}}{\psi_{\al_\O^i}(-t)^{\frac{\pi}{\Fi+\varepsilon}}}dt & \leq \int_A e^{p|\lambda|t} \frac{t^{\frac{\pi}{\Fi+\varepsilon}}}{e^{\frac{p|\lambda_0|\Fi}{\Fi+\varepsilon}t}} dt.
          \end{split}
    \end{equation*}
    Now, since $\frac{p|\lambda_0|\Fi}{\Fi+\varepsilon} \rightarrow p|\lambda_0| > p|\lambda|$ as $\varepsilon\rightarrow 0^+$, there exists some $\varepsilon>0$ sufficiently small such that the previous integral is convergent.

    \medskip

Hence, in order to deduce \eqref{integral convergente}, it remains to show that    
\begin{equation*}
              \int_B e^{p|\lambda|t}\frac{t^{\frac{\pi}{\Fi+\varepsilon}}}{\psi_{\al_\O^i}(-t)^{\frac{\pi}{\Fi+\varepsilon}}}dt < \infty,
              \end{equation*}
    for some $\varepsilon>0$. To this end, for each $\varepsilon>0$, write $\frac{\pi}{\Fi+\varepsilon} = \frac{\pi}{\Fi} - \varepsilon'$ for some $\varepsilon'>0$. Using this and the definition of the set $B$, we have
    \begin{equation*}
    \begin{split}
        \int_B e^{p|\lambda|t}\frac{t^{\frac{\pi}{\Fi+\varepsilon}}}{\psi_{\al_\O^i}(-t)^{\frac{\pi}{\Fi+\varepsilon}}}dt  = \int_B e^{p|\lambda_0|t}e^{-p\delta t}\frac{t^{\frac{\pi}{\Fi+\varepsilon}}\psi_{\al_\O^i}(-t)^{\varepsilon'}}{\psi_{\al_\O^i}(-t)^{\frac{\pi}{\Fi}}}dt \leq  \int_B e^{p|\lambda_0|t}e^{-p\delta t}\frac{t^{\frac{\pi}{\Fi+\varepsilon}}e^{\frac{p|\lambda|\Fi \varepsilon'}{\pi}t}}{\psi_{\al_\O^i}(-t)^{\frac{\pi}{\Fi}}}dt.
    \end{split}
              \end{equation*}
              At this point, choose $\varepsilon'>0$ such that $\frac{p|\lambda|\Fi \varepsilon'}{\pi}< \frac{\delta}{2},$ and notice that this choice fixes $\varepsilon>0$. Now, there exists $t_1\geq t_0$ such that $t^{\frac{\pi}{\Fi+\varepsilon}}\leq e^{\frac{\delta}{2}t}$, for all $t\geq t_1$. Thus, one has, for some $C>0$,
              $$
              \int_B e^{p|\lambda_0|t}e^{-p\delta t}\frac{t^{\frac{\pi}{\Fi+\varepsilon}}e^{\frac{p|\lambda|\Fi \varepsilon'}{\pi}t}}{\psi_{\al_\O^i}(-t)^{\frac{\pi}{\Fi}}}dt 
              \leq C+ \int_{B\cap[t_1,+\infty)} \frac{e^{p|\lambda_0|t}}{\psi_{\al_\O^i}(-t)^{\frac{\pi}{\Fi}}} dt,
              $$
        which is finite by hypothesis. This shows that $\lambda \in \esp$, and finishes the proof.
\end{proof}

Notice that Corollary \ref{corolario no convexo contiene sector} completely characterizes $\esp$ for those Koenigs domains $\Omega$ with associated defining functions $\psi_{\al_\O^i}$ and $\psi_{\b_\O^i}$ that grow rapidly enough. Indeed, if the following conditions are satisfied
$$
\lim\limits_{t\rightarrow+\infty} \frac{e^{at}}{\psi_{\al_\O^i}(-t)} = \lim\limits_{t\rightarrow+\infty} \frac{e^{at}}{\psi_{\b_\O^i}(t)} = 0
$$
for all $a>0$, the result holds.

\medskip
%\subsection{Case II: \texorpdfstring{$\bm{\alpha_\Omega^o + \beta_\Omega^o = 0}$}{c}}  
We pass now to the case $\alpha_\Omega^o + \beta_\Omega^o = 0$. The next result arises as an analogue to Theorem \ref{teorema no convexo por arriba} for this setting:

\begin{theorem}\label{teorema no convexo por arriba no contiene sector}
     Let $(\Fi_t)_{t>0}$ be a parabolic semigroup on $\D$. Suppose that $\Omega \subseteq S_p(\al_0,\beta_0)$ for some $p\in \R$, $\al_0\in [0,\pi),$ and $\beta_0\in [0,\pi)$, with $\alpha_0+\beta_0\leq\pi$. Suppose further that $\al_\O^o+\beta_\O^o =0$.
     \begin{enumerate}
         \item [(i)] Let $\lambda =|\lambda|e^{i\frac{\pi}{2}}$. Then, $\lambda \in \esp$ if one of the following conditions holds:
         \begin{enumerate}

\smallskip
         
             \item [(a)] $\dom(\psi_\O) = (a,+\infty)$ for some $a\in \R$.

             \smallskip

             \item [(b)] $\dom(\psi_\O) = (-\infty,b)$ for some $b \in \R$ and
             $$
             \int_{t_0}^\infty \frac{e^{p|\lambda|t}}{e^{\frac{\pi}{2}\frac{\psi_\O(-t)}{t}}} dt < \infty.
             $$
             \item [(c)] $\dom(\psi_\O) = \R$ and
              $$\int_{t_0}^\infty \frac{e^{p|\lambda|t}}{e^{\frac{\pi}{2}\frac{\widetilde{\psi_\O}(t)}{t}}} dt < \infty,
              $$
              where $\widetilde{\psi_\Omega}(t) := \min\{\psi_\O(-t),\psi_\O(t)\}$ for all $t\in\R$.
             \end{enumerate}

\smallskip

\smallskip

        \item [(ii)] Let $\lambda =|\lambda|e^{i\frac{3\pi}{2}}$. Then, $\lambda \in \esp$ if one of the following conditions holds:
         \begin{enumerate}
         
         \smallskip

             \item [(a)] $\dom(\psi_\O) = (-\infty,b)$ for some $b \in \R$. 

\smallskip
             
             \item [(b)] $\dom(\psi_\O) = (a,+\infty)$ for some $a\in \R$ and
             $$
             \int_{t_0}^\infty \frac{e^{p|\lambda|t}}{e^{\frac{\pi}{2}\frac{\psi_\O(t)}{t}}} dt < \infty.
             $$
             \item [(c)] $\dom(\psi_\O) = \R$ and
              $$\int_{t_0}^\infty \frac{e^{p|\lambda|t}}{e^{\frac{\pi}{2}\frac{\widetilde{\psi}_\O(t)}{t}}} dt < \infty,
              $$
              where $\widetilde{\psi_\Omega}(t) := \min\{\psi_\O(-t),\psi_\O(t)\}$ for all $t\in\R$.
         \end{enumerate}
     \end{enumerate}
\end{theorem}
\begin{proof}
    As usual, we prove $(i)$. The proof of $(ii)$ is analogous.

First, observe that, if $\dom(\psi_\O) = (a,+\infty)$, an identical argument to the one used in the beginning of the proof of Theorem \ref{teorema no convexo por arriba}, along with Proposition \ref{proposicion no convexo defining functions no contiene sector}, yields that $\lambda \in \esp$.

Now, assume that $\dom(\psi_\O) = (-\infty,b)$, and consider $\Omega_{\frac\pi2}= e^{i\frac\pi2}\Omega.$ In such a case, by Corollary \ref{corolario formula norma}
$$
\norm{e^{\lambda h}}^p \lesssim C+ \int_{t_0}^\infty e^{p|\lambda|t}\omega(0, E_t,\Omega_\frac\pi2)dt,
    $$
    where $E_t = \{\xi \in \partial \Omega_\frac\pi2: \PR(\xi) >t\}$ and $C>0$.

By the assumption on $\dom(\psi_\O)$, $\Omega_\frac\pi2$ is contained in the half-plane $b+\C_+$. Since $0\in \Omega,$ we have that $b<0$. Now, define the following vertical strip
$$
B_t = \{z\in \C: |\PR(z)|<t\}, \quad t\geq t_0.
$$
 Define also
$$
F_t^+:=\{t+ix : x> \psi_\O(-t)\} \quad \text{and} 
\quad 
F_t^-:=\{-t+ix : x> \psi_\O(-t)\}.
$$

Without loss of generality, we may assume that for all $t\geq t_0$, $F_t^-$ is contained in $\C\setminus(b+\C_+).$ Now, for all $t\geq t_0$, by the Maximum  Principle, it follows that
$$
\omega(0,E_t,\O_\frac\pi2) \leq \omega(0,F^+_t,B_t)
$$
for all $t\geq t_0$. Since $F_t^-$ is the reflection of $F_t^+$ with respect to the imaginary axis, and $B_t$ is symmetric with respect to it, we have that $$\omega(0,F^+_t,B_t) = \frac{1}{2}\omega(0,F^+_t\cup F^-_t,B_t).$$
At this point, it is enough to apply Proposition \ref{LTecnico:banda vertical} to deduce the result.

\medskip
It remains to prove the case $\dom(\psi_\O) = \R$. In order to perform a similar construction as the one carried out above, we need to symmetrize the rotated domain $\O_\frac\pi2$. First, observe that the function $\widetilde{\psi_\Omega}$ is the minimum of two upper semicontinuous functions, so it is also upper semicontinuous. Then, $\widetilde{\psi_\Omega}$ defines a parabolic Koenigs domain $\widetilde{\O}$, with Koenigs function $\widetilde{h}$, that is symmetric with respect to the real axis and satisfies
    $$
    \Omega \subset \widetilde{\Omega}.
    $$
    Now, $\Fi = \widetilde{h}^{-1}\circ h$ is a holomorphic self-map of $\D$, so the associated composition operator $C_\Fi$ is bounded on $H^p$. Hence, we have
    $$
    \norm{e^{\lambda h}} = \norm{C_\Fi e^{\lambda \widetilde{h}}} \leq \norm{C_\Fi}\norm{e^{\lambda \widetilde{h}}}.
    $$
    Now, again by Corollary \ref{corolario formula norma}
    $$
    \norm{e^{\lambda \widetilde{h}}}^p\lesssim C+ \int_{t_0}^\infty e^{p|\lambda|t}\omega(0, E_t^+,\widetilde{\O}_{\frac\pi2})dt,
    $$
    where $E_t^+ = \{\xi \in \partial \widetilde{\O}_\frac\pi2: \PR(\xi) >t\}$ and $C>0$. Likewise, define

    $$
    E_t^- = \{\xi \in \partial \widetilde{\O}_{\frac\pi2}: \PR(\xi) <-t\} $$
    and observe that $E_t^-$ is the reflection with respect to the imaginary axis of $E_t^+$.
    
    At this point, we may consider again $B_t$ as defined above, and the boundary subsets
    $$
    \widetilde{F^+_t}:= \{t+ix: x > \widetilde{\psi_\O}(t)\};\quad
    \widetilde{F^-_t}:= \{-t+ix: x > \widetilde{\psi_\O}(t)\}.
    $$
    Now, since $\widetilde{\O}$ is symmetric with respect to the imaginary axis, one has for all $t\geq t_0$
    $$
    \omega(0,E^+_t,\widetilde{\O_\frac\pi2}) = \frac{1}{2} \omega(0,E^+_t\cup E_t^-,\widetilde{\O_\frac\pi2}).
    $$
    By the Maximum Principle
    $$
    \omega(0,E^+_t\cup E_t^-,\widetilde{\O_\frac\pi2}) \leq \omega(0,\widetilde{F^+_t}\cup \widetilde{F_t^-},B_t)
    $$
    for all $t\geq t_0$, and again an application of Proposition \ref{LTecnico:banda vertical} yields the result.
\end{proof}

\begin{corollary}\label{corolario no convexo no contiene sector}
     Let $(\Fi_t)_{t>0}$ be a parabolic semigroup on $\D$. Suppose that $\Omega \subseteq S_p(\al_0,\beta_0)$ for some $p\in \R$, $\al_0\in [0,\pi)$, and $\beta_0\in [0,\pi)$, with $\alpha_0+\beta_0\leq\pi$. Suppose further that $\al_\O^o+\beta_\O^o =0$. Assume one of the following three conditions holds:
     \begin{enumerate}
         \item [(i)] $\dom(\psi_\O) = (a,+\infty)$ and
          $$
             \int_{t_0}^\infty \frac{e^{p|\lambda|t}}{e^{\frac{\pi}{2}\frac{\psi_\O(t)}{t}}} dt < \infty\quad \text{for all $|\lambda|>0$.}
             $$             
        \item [(ii)] $\dom(\psi_\O) = (-\infty,b)$ and 
         $$
             \int_{t_0}^\infty \frac{e^{p|\lambda|t}}{e^{\frac{\pi}{2}\frac{\psi_\O(-t)}{t}}} dt < \infty \quad \text{for all $|\lambda|>0$.}
             $$
        \item [(iii)] $\dom(\psi_\O) = \R$ and
         $$
             \int_{t_0}^\infty \frac{e^{p|\lambda|t}}{e^{\frac{\pi}{2}\frac{\widetilde{\psi_\O}(t)}{t}}} dt < \infty\quad \text{for all $|\lambda|>0$.}
             $$
        \end{enumerate}
      Then, $$
      \esp = \left\{re^{i\theta}: r\geq 0, \frac{\pi}{2}\leq \theta \leq \frac{3\pi}{2}\right\}.
      $$
\end{corollary}
\begin{proof}
    The result follows as a consequence of both Theorem \ref{teorema contenciones general} and Theorem \ref{teorema no convexo por arriba no contiene sector}.
\end{proof}
Again, Corollary \ref{corolario no convexo no contiene sector} may be applied for defining functions $\psi_\O$ that grow fast enough. For instance, by assuming that there exists $\gamma >2$ such that $\psi_\Omega(t) \geq C |t|^\gamma$  for all $|t|\geq t_0$, the result holds.

\subsection{An illustrating particular case}\label{subseccion convexo} To finish this section, we consider convex Koenigs domains $\Omega$ and obtain necessary conditions for $e^{\lambda h}$ to belong to $H^p$. Such conditions, together with Theorem \ref{teorema no convexo por arriba}, completely characterize $\esp$, up to two points.

\medskip

Let $(\phi_t)_{t\geq 0}$ be a parabolic semigroup such that its Koenigs domain $\Omega$ is convex. This yields that $\Omega$ is contained in $S_p(\al_0,\beta_0)$ for some $p\in \C$ and $\al_0,\beta_0\in [0,\pi]$, where $0<\al_0+\beta_0\leq \pi$. It is straightforward to check that $\al_\O^i = \al_\O^o$ and $\beta_\O^i = \beta_\O^o$. For simplicity, let us denote, respectively, $\al_\O := \al_\O^i = \al_\O^o$ and $\beta_\O := \beta_\O^i = \beta_\O^o$. In this situation, Theorem \ref{teorema contenciones general} gives that

$$  \sigma_p(\Delta\mid_{H^p})\supseteq\left\{re^{i\theta}: r\geq 0, \ \frac{\pi}{2}+\al_\O < \theta < \frac{3\pi}{2}-\beta_\O \right\} $$ 
      $$
           \sigma_p(\Delta\mid_{H^p})\subseteq \left\{re^{i\theta}: r\geq 0, \ \frac{\pi}{2}+\al_\O \leq \theta \leq \frac{3\pi}{2}-\beta_\O \right\}.
            $$
			In particular, if $\al_\Omega +\beta_\Omega = \pi$, $\Omega$ is a half-plane, so by Corollary \ref{corolario contenciones sector angular},
			$\sigma_p(\Delta\mid_{H^p}) =\ell\left(\frac{\pi}{2}+ \al_\Omega\right)$. Regarding whether $\lambda=|\lambda|e^{i\left(\frac{\pi}{2}+\al_\O\right)}$ or $\mu =|\mu|e^{i\left(\frac{3\pi}{2}-\beta_\O\right)}$ lies in $\esp$, we have:

\begin{theorem}\label{teorema caso convexo contiene sector}
		Let $(\Fi_t)_{t\geq0}$ be a parabolic semigroup on $\D$ with convex  Koenigs domain $\Omega$. Assume that $\Omega$ contains an angular sector and $0<\al_\O+\b_\O<\pi$.
		\begin{enumerate}
			\item [(i)] Let $\lambda =|\lambda|e^{i\left(\frac{\pi}{2}+\al_\O\right)}$. Then, $\lambda \notin \esp$ provided that  $\dom(\psi_{\O_\al})=\R$ and
				\begin{equation}\label{integral acotacion por abajo convexo}
					\int_{t_0}^\infty e^{p|\lambda|t}\left(\frac{1}{\psi_{\O_\al}(-t)}\right)^{\frac{\pi}{\al_\O+\beta_\O}}dt = \infty.
				\end{equation}
			\item [(ii)] Let $\mu =|\mu|e^{i\left(\frac{3\pi}{2}-\beta_\O\right)}.$ Then, $\mu \notin \esp$ provided that  $\dom(\psi_{\O_\b})=\R$ and
                \begin{equation}\label{integral acotacion por abajo convexo beta}
                    \int_{t_0}^\infty e^{p|\mu|t}\left(\frac{1}{\psi_{\O_\b}(t)}\right)^{\frac{\pi}{\al_\O+\beta_\O}}dt = \infty.
                \end{equation}

		\end{enumerate}
	\end{theorem}
\begin{proof}
    As usual, we just prove $(i)$. Assume that $\dom(\psi_{\O_\al}) = \R$. By the convexity of $\Omega$, ${\O_\al}$ is itself a convex Koenigs domain, so $\psi_{\O_\al}$ is a continuous convex function that coincides with the defining functions of such a domain. As a consequence of Fekete's Lemma (\cite[Theorem 7.6.1]{HP}) and the minimality of $\al_\O$, it follows that  
    \begin{equation}\label{limite}
        \lim\limits_{t\rightarrow \infty} \frac{\psi_{\Omega_{\al}}(-t)}{t}=+\infty.
    \end{equation}
    
    Now, given $\lambda =|\lambda|e^{i\left(\frac{\pi}{2}+\al_\O\right)}$, our goal is to find constants $C_1, t'$ such that
	\begin{equation}\label{acotacion por abajo}
		\norm{e^{\lambda h}}_{H^p} \gtrsim C_1+\int_{t'}^\infty e^{p|\lambda|t}\left(\frac{1}{\psi_{\O_\a}(-t)}\right)^{\frac{\pi}{\al_\O+\beta_\O}}dt,
	\end{equation}
	which, by hypothesis \eqref{integral acotacion por abajo convexo}, will diverge and $\lambda$ will fail to belong to the point spectrum, as desired. By \eqref{expresion norma medida armonica}, it is enough to obtain a proper lower bound for  $\omega(0,E_t, \Omega_\sigma),$ 	where  $\sigma = \frac{\pi}{2}+\al_\O$, $\Omega_\sigma = e^{i\sigma}\Omega$ and $E_t = \{ \xi \in \Omega_\sigma: \Re(\xi) >t\}$. 
	
	Let $t_0$ be as in Proposition \ref{proposicion no convexo defining functions}. For all $t\geq t_0$, define $\theta_t = \arctan\left( \frac{\psi_{\O_\a}(-t)}{t}\right).$ Now, consider the rotated angular sectors 	$S_t = e^{i\theta_t}S(0,\al_\O+\beta_\O).$  We have $\partial S_t = \ell(\theta_t)\cup \ell(\al_\O+\beta_\O+\theta_t).$ By \eqref{limite}, $\theta_t \rightarrow \frac{\pi}{2}$ as $t\rightarrow +\infty$, so there exists $t_1\geq t_0$ such that $\ell(\al_\O+\beta_\O+\theta_t)\subset \Omega_\s$ for all $t\geq t_1.$ On the other hand, $\ell(\theta_t)\cap \partial \Omega_\s = \{z_t\},$ where $z_t = t+i\psi_{\O_\a}(-t)$ and  
	
	$$
	\ell(\theta_t)\cap (\C\setminus \Omega_\s) = F_t = \{ \xi \in \partial S_t : \PR(\xi) > t\}.
	$$
	
	Now, for each $t\geq t_1$, define $\zeta_t = e^{i(\theta_t + \frac{\al_\O+\beta_\O}{2})}$, and note that it is the point of modulus $1$ in the bisector of $S_t$. Clearly, there exists a compact subset $K\subset \Omega_\sigma$ such that $\zeta_t \in K$ for all $t\geq t_1$, so by Theorem \ref{teorema distancia hiperbolica} and the continuity of the hyperbolic distance, there exists a constant $C>0$ (not depending on $t$) such that
	$$
	\omega(0,E_t,\Omega_\sigma) \geq C \omega(\zeta_t,E_t,\O_\s).
	$$
	At this point, we can apply the maximum modulus principle so that
	$$
	\omega(\zeta_t,E_t,\O_\s) \geq \omega(\zeta_t, F_t, S_t).
	$$
	Moreover, by the choice of $\zeta_t$, we can apply the second part of Corollary \ref{calculo medida armonica} to obtain, by \eqref{lemabisec}, that
	
	$$
	\omega(\zeta_t, F_t, S_t) = \frac{1}{2\pi} \arctan\left(\frac{2 (t^2+\psi_{\O_\a}(-t)^2)^{\frac{\pi}{2(\al_\O+\beta_\O)}} }{(t^2+\psi_{\O_\a}(-t)^2)^{\frac{\pi}{(\al_\O+\beta_\O)}}-1}         \right)\cdot
	$$
	Now, there exists $t_2 \geq t_1$ such that for all $t\geq t_2$
	
	\begin{equation*}
		\begin{split}
			\arctan\left(\frac{2 (t^2+\psi_{\O_\a}(-t)^2)^{\frac{\pi}{2(\al_\O+\beta_\O)}} }{(t^2+\psi_{\O_\a}(-t)^2)^{\frac{\pi}{(\al_\O+\beta_\O)}}-1}         \right) 
			 \gtrsim \frac{1}{\psi_{\O_\a}(-t)^{\frac{\pi}{\al_\O+\beta_\O}}}\cdot
		\end{split}
	\end{equation*}
	This concludes the proof. 
\end{proof}
\begin{remark}
  We finish this section by highlighting a consequence of Theorem \ref{teorema caso convexo contiene sector}. Under its hypotheses, the convergence of the integral in \eqref{integral acotacion por abajo convexo} characterizes whether $\lambda =|\lambda|e^{i\left( \frac{\pi}{2}+\al_\O\right)}$ belongs to $\esp$, except for, at most, one point. An analogous characterization holds for $\mu = |\mu|e^{i\left( \frac{3\pi}{2}-\b_\O\right)}$ using the integral in \eqref{integral acotacion por abajo convexo beta}. We just reason for $\lambda$, the arguments for $\mu$ being identical.   

\ 
Note that, if the integral in \eqref{integral acotacion por abajo convexo} converges for some $|\lambda_0| >0$, the arguments in the proof of Corollary \ref{corolario no convexo contiene sector} ensure that \eqref{integral por arriba} converges for every $|\lambda|<|\lambda_0|$, for some appropriate $\varepsilon>0$ depending on $|\lambda|$. Consequently, $\lambda = |\lambda|e^{i\left( \frac{\pi}{2}+\al_\O\right)}$ belongs to $\esp$ whenever $|\lambda|<|\lambda_0|$. This leads to three mutually exclusive scenarios:
   \begin{enumerate}
       \item [(i)] If the integral in \eqref{integral acotacion por abajo convexo} converges for all $|\lambda|>0$, then $\lambda \in \esp$ for all $|\lambda|>0$. 
       \item [(ii)] If the integral in \eqref{integral acotacion por abajo convexo} diverges for all $|\lambda|>0$, then $\lambda\notin \esp$ for any $|\lambda|>0$.
       \item [(iii)] If there exists $|\lambda_0|>0$ such that the integral in \eqref{integral acotacion por abajo convexo} converges for all $|\lambda|< |\lambda_0|$ and diverges for $|\lambda| > |\lambda_0|$, then $\lambda \in \esp$ if $|\lambda|<|\lambda_0|$ and $\lambda\notin \esp$ if $|\lambda|>|\lambda_0|$. In general, the convergence behaviour of the integral in \eqref{integral acotacion por abajo convexo} for $|\lambda_0|$ is not sufficient to decide whether $\lambda_0=|\lambda_0|e^{i\left( \frac{\pi}{2}+\al_\O\right)}$ belongs to $\esp$.
   \end{enumerate}
\end{remark}

\section{Spectral consequences for composition operators}\label{seccion espectro}
As it was pointed out in the introduction, if $(\Fi_t)_{t\geq 0}$ is a parabolic semigroup on $\D$, we can recover the point spectrum on $H^p$ of the composition operators $(C_{\Fi_t})_{t\geq 0}$ through the identity
\begin{equation}\label{identidad espectro puntual}
    \sigma_p(C_{\Fi_t}\mid_{H^p}) = \{e^{t\lambda} : \lambda \in \esp\} \qquad (t>0).
\end{equation}
In this section, we exploit the previous equality to completely characterize the spectrum $\sigma(C_{\Fi_t}\mid_{H^p})$ of these operators in some special cases. 

\medskip

Given a linear bounded operator $T:X\rightarrow X$ acting on a complex Banach space, we denote by $r(T)$ its spectral radius, which is given by
$$
r(T) = \sup\{|\lambda|:\lambda \in \sigma(T)\},
$$
where $\sigma(T)$ is the spectrum of $T$. Siskakis \cite[Corollary 3.2]{Tesis_Siskakis} characterized the spectral radius of composition operators $C_\phi$ acting on $H^p$ induced by a holomorphic self-map of the unit disc $\phi$ with Denjoy-Wolff point $1$. Indeed, he proved that $$r(C_\phi)=\phi'(1).$$
As a consequence, regarding the definition of parabolic semigroups, the following holds:
\begin{proposition}\label{proposicion radio espectral}
    Let $(\Fi_t)_{t\geq 0}$ be a parabolic semigroup on $\D$. Then,
    $r(C_{\Fi_t}) = 1$ for all $t>0$.
\end{proposition}
We can now state the main result of this section:

\begin{theorem}\label{teorema espectro}
Let $(\Fi_t)_{t\geq 0}$ be a parabolic semigroup on $\D$. Suppose that $\Omega \subseteq S_{p}(\al_0,\beta_0)$ for some $p\in \R$, $\al_0\in [0,\pi)$, and $\beta_0\in [0,\pi)$, where $0<\al_0+\beta_0\leq \pi$. Assume that either
\begin{enumerate}
    \item [(i)] $\al_\O^o =0$ and $\beta_\O^o <\pi$, or
    \item [(ii)] $\al_\O^o < \pi$ and $\b_\O^o = 0$.
\end{enumerate}
Then, $$\sigma(C_{\Fi_t}\mid_{H^p}) = \overline{\D}$$
for all $t>0$.
\end{theorem}
\begin{proof}
  First, by Proposition \ref{proposicion radio espectral}, $\sigma(C_{\Fi_t}\mid_{H^p}) \subseteq \overline{\D}$ for all $t>0$. Now, let us assume $(i)$, the reasoning for $(ii)$ is equivalent. By Theorem \ref{teorema contenciones general}, it follows that
    $$A:=\left\{re^{i\theta}: r\geq 0, \frac{\pi}{2}<\theta < \frac{3\pi}{2}-\beta_\O^o\right\} \subseteq \sigma_p(\Delta).$$
    Now, it is easy to show that $\exp(A) = \D\setminus \{0\}$, which along with \eqref{identidad espectro puntual} yields that
    $$
    \D\setminus \{0\}\subseteq \sigma(C_{\Fi_t}\mid_{H^p}) \subseteq \overline{\D}
    $$
    for all $t>0$. Recalling that $\sigma(C_{\Fi_t}\mid_{H^p})$ is closed, the result follows. 
\end{proof}

This result may be applied to describe the spectrum of composition operators associated to parabolic semigroups $(\Fi_t)_{t\geq 0}$ with positive hyperbolic step:

\begin{corollary}\label{corolario paso hiperbolico}
    Let $(\Fi_t)_{t\geq 0}$ be a parabolic semigroup on $\D$ of positive hyperbolic step. Then, 
    $$
    \T\subseteq \sigma(C_{\Fi_t}\mid_{H^p})\subseteq \overline{\D}.
    $$
    Moreover, if $\Omega$ is contained in an angular sector $S_p(\al,0)$ or $S_p(0,\beta)$ for some $p\in \C$ and $\al<\pi$, $\beta < \pi$, then 
    $$\sigma(C_{\Fi_t}\mid_{H^p})= \overline{\D}.$$
\end{corollary}
\begin{proof}
    Since $(\Fi_t)_{t\geq 0}$ is of positive hyperbolic step, $\Omega$ is contained in a horizontal half-plane. By Theorem \ref{espectro sector angular}, Proposition \ref{contencion espectro puntual}, and \eqref{identidad espectro puntual}, $\T\subseteq \sigma_p(C_{\Fi_t}\mid_{H^p})$. Proposition \ref{proposicion radio espectral} yields $\sigma(C_{\Fi_t}\mid_{H^p})\subseteq \overline{\D}.$ Now, assume $\Omega$ is contained in an angular sector $S_p(\al,0)$. The other case is equivalent. This implies that $\b_\O^o =0$ and $\al_\O^o >0$, so Theorem \ref{teorema espectro} yields the result.
\end{proof}
It is worth mentioning that in the case in which $(\Fi_t)_{t\in \R}$ is a parabolic group of automorphisms of $\D$, $\Omega$ is a horizontal half-plane, so it is of positive hyperbolic step and $\sigma(C_{\Fi_t}\mid_{H^p}) = \T$. However, for any other parabolic semigroup $(\Fi_t)_{t\geq 0}$ not consisting of automorphisms of $\D$, the operators $C_{\Fi_t}$ are not invertible, so $0 \in \sigma(C_{\Fi_t}\mid_{H^p})$. 

\medskip

\textbf{Acknowledgements.} We thank Professor Manuel D. Contreras for his careful reading of the manuscript and his valuable suggestions to improve it. We also acknowledge Professor Francisco J. Cruz Zamorano for sharing his insights on harmonic measure.

\bibliographystyle{siam}
\bibliography{bibliografianew}

@book{BCD,
  title={Continuous semigroups of holomorphic self-maps of the unit disc},
  author={Bracci, Filippo and Contreras, Manuel D and D{\'\i}az-Madrigal, Santiago},
  year={2020},
  publisher={Springer}
}

@article{BP,
  title={Semigroups of analytic functions and composition operators},
  author={Berkson, Earl and Porta, Horacio},
  journal={Michigan Mathematical Journal},
  volume={\textbf{25}},
  number={1},
  pages={101--115},
  year={1978},
  publisher={University of Michigan, Department of Mathematics}
}

@article{Betsakos,
  title={On the eigenvalues of the infinitesimal generator of a semigroup of composition operators},
  author={Betsakos, Dimitrioss},
  journal={Journal of Functional Analysis},
  volume={\textbf{273}},
  number={7},
  pages={2249--2274},
  year={2017},
  publisher={Elsevier}
}

@article{BeCD,
  title={On the rate of convergence of semigroups of holomorphic functions at the {D}enjoy--{W}olff point},
  author={Betsakos, Dimitrios and Contreras, Manuel D and Diaz-Madrigal, Santiago},
  journal={Revista Matematica Iberoamericana},
  volume={\textbf{36}},
  number={6},
  pages={1659--1686},
  year={2020}
}

@misc{BCZ,
      title={On the rates of convergence of orbits in semigroups of holomorphic functions}, 
      author={Dimitrios Betsakos and Francisco J. Cruz-Zamorano and Konstantinos Zarvalis},
      year={2025},
      note= {	arXiv:2503.20388, \url{https://arxiv.org/abs/2503.20388}      
}}

@article{BGY,
  title={Complete frequencies for {K}oenigs domains},
  author={Bracci, Filippo and Gallardo-Guti{\'e}rrez, Eva A. and Yakubovich, Dmitry},
  journal={J. Eur. Math. Soc. (in press)},
  year={2025}
}

@article {CS,
    AUTHOR = {Caughran, James G. and Schwartz, Howard J.},
     TITLE = {Spectra of {C}ompact {C}omposition {O}perators},
   JOURNAL = {Proc. Amer. Math. Soc.},
  FJOURNAL = {Proceedings of the American Mathematical Society},
    VOLUME = {\textbf{51}},
      YEAR = {1975},
     PAGES = {127--130},
      ISSN = {0002-9939,1088-6826},
   MRCLASS = {47B37 (30A18)},
  MRNUMBER = {377579},
MRREVIEWER = {Joel\ H.\ Shapiro},
       DOI = {10.2307/2039858},
       URL = {https://doi.org/10.2307/2039858},
}

@article {CCKR,
    AUTHOR = {Contreras, Manuel D. and Cruz-Zamorano, Francisco J. and
              Kourou, Maria and Rodr\'iguez-Piazza, Luis},
     TITLE = {On the {H}ardy number of {K}oenigs domains},
   JOURNAL = {Anal. Math. Phys.},
  FJOURNAL = {Analysis and Mathematical Physics},
    VOLUME = {\textbf{14}},
      YEAR = {2024},
    NUMBER = {6},
     PAGES = {Paper No. 119, 21},
      ISSN = {1664-2368,1664-235X},
   MRCLASS = {30D05 (30C85 30H10 37F99 39B32)},
  MRNUMBER = {4813211},
MRREVIEWER = {Konstantinos\ Zarvalis},
       DOI = {10.1007/s13324-024-00981-4},
       URL = {https://doi.org/10.1007/s13324-024-00981-4},
}

@article {CD,
    AUTHOR = {Contreras, Manuel D. and D\'iaz-Madrigal, Santiago},
     TITLE = {Analytic flows on the unit disk: angular derivatives and
              boundary fixed points},
   JOURNAL = {Pacific J. Math.},
  FJOURNAL = {Pacific Journal of Mathematics},
    VOLUME = {\textbf{222}},
      YEAR = {2005},
    NUMBER = {2},
     PAGES = {253--286},
      ISSN = {0030-8730,1945-5844},
   MRCLASS = {30C20 (30D05 30F45 37F10)},
  MRNUMBER = {2225072},
MRREVIEWER = {Alexander\ Vasil\cprime ev},
       DOI = {10.2140/pjm.2005.222.253},
       URL = {https://doi.org/10.2140/pjm.2005.222.253},
}

@article{Cowen,
  title={Composition operators on ${H}^2$},
  author={Cowen, Carl C},
  journal={Journal of Operator Theory},
  pages={77--106},
  year={1983}
}

@book {CM,
    AUTHOR = {Cowen, Carl C. and MacCluer, Barbara D.},
     TITLE = {Composition operators on spaces of analytic functions},
    SERIES = {Studies in Advanced Mathematics},
 PUBLISHER = {CRC Press, Boca Raton, FL},
      YEAR = {1995},
     PAGES = {xii+388},
      ISBN = {0-8493-8492-3},
   MRCLASS = {47B38 (30D55 46E15)},
  MRNUMBER = {1397026},
MRREVIEWER = {John\ N.\ McDonald},
}

@article {Deddens,
    AUTHOR = {Deddens, James A.},
     TITLE = {Analytic {T}oeplitz and {C}omposition {O}perators},
   JOURNAL = {Canadian J. Math.},
  FJOURNAL = {Canadian Journal of Mathematics. Journal Canadien de
              Math\'ematiques},
    VOLUME = {\textbf{24}},
      YEAR = {1972},
     PAGES = {859--865},
      ISSN = {0008-414X,1496-4279},
   MRCLASS = {47B35},
  MRNUMBER = {310691},
MRREVIEWER = {C.\ R.\ Putnam},
       DOI = {10.4153/CJM-1972-085-8},
       URL = {https://doi.org/10.4153/CJM-1972-085-8},
}

@book{Duren,
  title={Theory of $H^p$ {S}paces},
  author={Duren, Peter L},
  year={1970},
  publisher={Pure Appl. Math., Vol. 38, Academic Press, New York-London}
}

@book{EN,
  title={A short course on operator semigroups},
  author={Engel, Klaus-Jochen and Nagel, Rainer},
  year={2006},
  publisher={Springer Science \& Business Media}
}

@book{GaMA,
  title={Harmonic measure},
  author={Garnett, John B and Marshall, Donald E},
  number={2},
  year={2005},
  publisher={Cambridge University Press}
}

@article{GD,
    AUTHOR = {Gonz\'alez-Do\~na, F. Javier},
     TITLE = {On commutants of composition operators embedded into
              {$C_0$}-semigroups},
   JOURNAL = {Rev. R. Acad. Cienc. Exactas F\'is. Nat. Ser. A Mat. RACSAM},
  FJOURNAL = {Revista de la Real Academia de Ciencias Exactas, F\'isicas y
              Naturales. Serie A. Matematicas. RACSAM},
    VOLUME = {119},
      YEAR = {2025},
    NUMBER = {4},
     PAGES = {Paper No. 102, 21},
      ISSN = {1578-7303,1579-1505},
   MRCLASS = {47B33 (47A15 47D06)},
  MRNUMBER = {4938278},
       DOI = {10.1007/s13398-025-01770-9},
       URL = {https://doi.org/10.1007/s13398-025-01770-9},
}

@book {HP,
    AUTHOR = {Hille, Einar and Phillips, Ralph S.},
     TITLE = {Functional {A}nalysis and {S}emi-groups},
    SERIES = {American Mathematical Society Colloquium Publications},
    VOLUME = {31},
      NOTE = {rev. ed},
 PUBLISHER = {American Mathematical Society, Providence, RI},
      YEAR = {1957},
     PAGES = {xii+808},
   MRCLASS = {46.2X},
  MRNUMBER = {89373},
MRREVIEWER = {M.\ H.\ Stone},
}

@article {Kamowitz,
    AUTHOR = {Kamowitz, Herbert},
     TITLE = {The spectra of composition operators on {$H\sp{p}$}},
   JOURNAL = {J. Funct. Anal.},
  FJOURNAL = {Journal of Functional Analysis},
    VOLUME = {\textbf{18}},
      YEAR = {1975},
     PAGES = {132--150},
      ISSN = {0022-1236},
   MRCLASS = {47B37},
  MRNUMBER = {407645},
MRREVIEWER = {Joel\ H.\ Shapiro},
       DOI = {10.1016/0022-1236(75)90021-x},
       URL = {https://doi.org/10.1016/0022-1236(75)90021-x},
}

@article {Nordgren,
    AUTHOR = {Nordgren, Eric A.},
     TITLE = {Composition operators},
   JOURNAL = {Canadian J. Math.},
  FJOURNAL = {Canadian Journal of Mathematics. Journal Canadien de
              Math\'ematiques},
    VOLUME = {20},
      YEAR = {1968},
     PAGES = {442--449},
      ISSN = {0008-414X,1496-4279},
   MRCLASS = {47.25},
  MRNUMBER = {223914},
MRREVIEWER = {E.\ R.\ Deal},
       DOI = {10.4153/CJM-1968-040-4},
       URL = {https://doi.org/10.4153/CJM-1968-040-4},
}

@article {Poggi-Corradini,
    AUTHOR = {Poggi-Corradini, Pietro},
     TITLE = {The {H}ardy class of {K}oenigs maps},
   JOURNAL = {Michigan Math. J.},
  FJOURNAL = {Michigan Mathematical Journal},
    VOLUME = {\textbf{44}},
      YEAR = {1997},
    NUMBER = {3},
     PAGES = {495--507},
      ISSN = {0026-2285,1945-2365},
   MRCLASS = {47B38 (30D35 30D55 47A10)},
  MRNUMBER = {1481115},
MRREVIEWER = {William\ Thomas\ Ross},
       DOI = {10.1307/mmj/1029005784},
       URL = {https://doi.org/10.1307/mmj/1029005784},
}

@book{Ran,
  title={Potential {T}heory in the {C}omplex {P}lane},
  author={Ransford, Thomas},
  volume={\textbf{28}},
  year={1995},
  publisher={Cambridge University Press}
}

@book {Tesis_Siskakis,
    AUTHOR = {Siskakis, Aristomenis Georgios},
     TITLE = {Semigroups of composition operators and the Cesàro operator on $H^p(\D)$ },
      NOTE = {Thesis (Ph.D.)--University of Illinois at Urbana-Champaign},
 PUBLISHER = {ProQuest LLC, Ann Arbor, MI},
      YEAR = {1985},
     PAGES = {82},
   MRCLASS = {99-05},
  MRNUMBER = {2634547},
       URL =
              {http://gateway.proquest.com/openurl?url_ver=Z39.88-2004&rft_val_fmt=info:ofi/fmt:kev:mtx:dissertation&res_dat=xri:pqdiss&rft_dat=xri:pqdiss:8600315},
}

@article{Siskakis,
  title={Semigroups of composition operators on spaces of analytic functions, a review},
  author={Siskakis, Aristomenis G},
  journal={Contemporary Mathematics},
  volume={\textbf{213}},
  pages={229--252},
  year={1998},
  publisher={Providence, RI: American Mathematical Society}
}

@article{kourouetal,
AUTHOR = {Kourou, Maria and Theodosiadis, Eleftherios K. and Zarvalis,
              Konstantinos},
     TITLE = {Eigenvalues for infinitesimal generators of semigroups of
              composition operators},
   JOURNAL = {J. Funct. Anal.},
  FJOURNAL = {Journal of Functional Analysis},
    VOLUME = {291},
      YEAR = {2026},
    NUMBER = {6},
     PAGES = {Paper No. 111554},
     SSN = {0022-1236,1096-0783},
   MRCLASS = {47B33 (30C45 30D05 30H99 47D06)},
  MRNUMBER = {5071388},
       DOI = {10.1016/j.jfa.2026.111554},
       URL = {https://doi.org/10.1016/j.jfa.2026.111554},
}

\end{document}